\documentclass[12pt]{article}
\usepackage{fullpage}
\usepackage{amsmath,amssymb,amsthm}
\usepackage{graphicx,epsfig}

\newcommand{\aaa}{\alpha}
\newcommand{\eps}{\varepsilon}

\renewcommand{\ge}{\geqslant}

\newtheorem{theorem}{Theorem}
\newtheorem{proposition}{Proposition}
\newtheorem{corollary}{Corollary}
\newtheorem{lemma}{Lemma}
\newtheorem*{claima}{Claim}
\newtheorem{claim}{Claim}
\newenvironment{claimm}{\setcounter{claim}0\begin{claim}}{\end{claim}}

\theoremstyle{remark}
\newtheorem{definition}{Definition}
\newtheorem{remark}{Remark}
\newtheorem{example}{Example}

\usepackage{amscd}

\renewcommand{\SS}{\mathcal S}
\newcommand{\T}{\mathcal T}
\newcommand{\RR}{\mathcal R}
\newcommand{\F}{z}
\newcommand{\HH}{h}
\newcommand{\arrows}{A2}

\newcommand{\lthree}{\overleftarrow 3}
\newcommand{\rthree}{\overrightarrow 3}
\newcommand{\lfour}{\overleftarrow 4}
\newcommand{\rfour}{\overrightarrow 4}
\newcommand{\lsix}{\overleftarrow 6}
\newcommand{\rsix}{\overrightarrow 6}
\newcommand{\lfive}{\overleftarrow 5}
\newcommand{\rfive}{\overrightarrow 5}

\newcommand{\ltr}{\overleftarrow 3}
\newcommand{\rtr}{\overrightarrow 3}
\newcommand{\lsx}{\overleftarrow 6}
\newcommand{\rsx}{\overrightarrow 6}
\newcommand{\lfr}{\overleftarrow 4}
\newcommand{\rfr}{\overrightarrow 4}
\newcommand{\lfv}{\overleftarrow 5}
\newcommand{\rfv}{\overrightarrow 5}

\newcommand{\sss}{S}

\title{On the structure of Ammann A2 tilings\thanks{The work
was in part supported by was partially supported
by the Russian Academic Excellence Project `5-100',
by the RFBR grant 16-01-00362,  and the ANR grant
RaCAF ANR-15-CE40-0016-01}}

\author{Bruno Durand\thanks{LIRMM,
 Universit\'e de Montpellier II}, 
Alexander Shen\footnotemark[2],
Nikolay Vereshchagin\thanks{Moscow State University
  and
  National Research University
Higher School of Economics}}

\date{}

\begin{document}

\maketitle

\begin{abstract}
We establish a structure theorem 
for the family of Ammann A2 tilings of the plane.
Using that theorem
we show that 
every Ammann A2 tiling 
is self-similar in the sense of [B. Solomyak,
Nonperiodicity implies unique composition for
self-similar translationally finite tilings, Discrete and Computational
Geometry 20 (1998) 265-279].
By the same techniques
we show that Ammann A2 tilings are not robust in the
sense of
[B. Durand, A. Romashchenko, A. Shen.
Fixed-point tile sets and their applications,
Journal of Computer and System Sciences,
78:3 (2012) 731--764].
\end{abstract}

\section{Introduction}
\label{intro}

There is a 
non-convex hexagon with right angles that has the following property. 
It can be cut into two similar hexagons so that the scaling factors
are equal to $\psi$ and $\psi^2$, where 
$\psi<1$ (see Fig.~\ref{f2}).
\begin{figure}[t]
        \begin{center}
\includegraphics[scale=0.7]{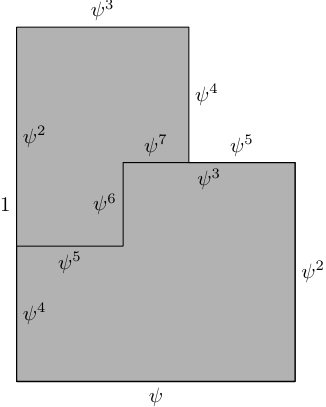}
        \end{center}
\caption{\small Cutting the Golden Bee into similar parts.}\label{f2}
\end{figure}
As the area of the original hexagon is equal
to the sum of areas of the parts, the number $\psi$ satisfies the
equation
        $$
\psi^4 + \psi^2 = 1.
        $$
That is, $\psi$ is the square root of the golden ratio:
$\psi=\sqrt{\frac{\sqrt5-1}{2}}$.

The numbers on the sides in Fig.~\ref{f2}
indicate their lengths, which are powers of $\psi$.
Using the equation $\psi^{n+2}+\psi^{n+4}=\psi^n$, 
it is easy to verify 
that the picture is consistent.
Following Scherer~\cite{sherer}, we will call any hexagon that is similar to 
that on Fig.~\ref{f2} a \emph{Golden Bee}.
The \emph{size} of a Golden Bee is defined as
the length of its largest side.

We fix a positive real $d$ and consider 
Golden Bees of sizes $d$  and $\psi d$ as tiles (see Fig.~\ref{f89}),
\begin{figure}[t]
       \begin{center}
\includegraphics[scale=0.5]{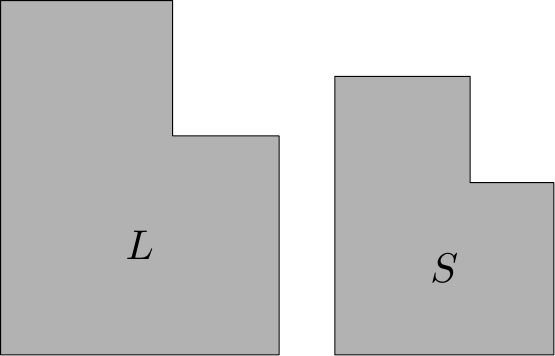}
        \end{center}
\caption{Ammann tiles.}\label{f89}
\end{figure} 
called \emph{large} and \emph{small} $d$-tiles.
Tilings of the plane or its parts by these two tiles will be called
\emph{$d$-tilings}.
If a $d$-tile $P$ is cut into small and large $d\psi$-tiles,
as shown  in Fig.~\ref{f2}, then we
call the large and small parts
the  \emph{son} and the \emph{daughter}
of $P$, respectively. We also call the small part the 
\emph{sister} of the large part
and call the large  part the 
\emph{brother} of the small part.

For a
$d$-tiling $T$ we denote by $\sigma T$ the $d\psi$-tiling
obtained from $T$ by the substitution shown on Fig.~\ref{f58}:
\begin{figure}[t]
\begin{center}
\includegraphics[scale=1]{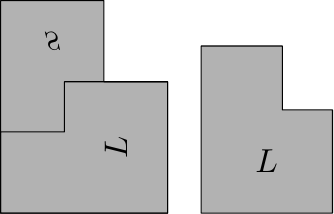}
        \end{center}
 \caption{\small Substitution: every large tile (on the left)
is cut into small and large tiles,
and every small tile (on the right) becomes a large tile.}\label{f58}
\end{figure}
we cut each large $d$-tile in two smaller tiles,
as shown on Fig.~\ref{f2}, and keep small $d$-tiles intact.
Small $d$-tiles thus become large $d\psi$-tiles of the resulting
$d\psi$-tiling.
It is not hard to prove that
$\sigma$ in an
\emph{injective}
mapping.

Gr\"unbaum and
Shephard~\cite{GS} considered three 
families of $d$-tilings  of the plane.
Those families are defined by means of rules 
governing how one may attach tiles to each other when tiling the plane
\cite[Fig. 10.4.1(a)]{GS}, 
\cite[Fig. 10.4.1(c)]{GS} and \cite[Fig. 10.4.1(d)]{GS}.
All the three families are called 
``A2''
and are attributed to
Robert Ammann.
The common name for these three families assumes
that the families coincide. This is indeed true
but is not evident and is not proven in~\cite{GS} or elsewhere
(we know that the families coincide
from a personal communication of Korotin~\cite{korotin}).
Yet another similar rule 
was introduced by Akiyama~\cite{A}.
One can show~\cite{korotin} that the family of tilings 
satisfying Akiyama's rule coincides with the A2 family.

All the three rules of~\cite{GS}, as well as Akiyama's rule,
imply the following  \emph{unique composition property}:
\begin{quote}
  {\em
    For  
each $d\psi$-tiling $T'$ from A2 there is a (unique)  
$d$-tiling $T$ in A2 such that $T'=\sigma T$.}
\end{quote}
A well known ``folklore'' theorem (see~\cite[Theorem 10.1.1]{GS})
states that the unique composition property 
implies that all the tilings in the family are 
non-periodic\footnote{\emph{Sketch of proof.} 
Assume that a tiling $T'$ from the family 
has a non-zero period
$t$, that is $T'+t=T'$. 
Let $T$ be the (unique) tiling from the family
such that $T'$ can be obtained from $T$ by the substitution.
Then $t$ is a period 
of the tiling $T$ as well. Indeed, since substitution and
shift commute, the substitution applied 
to the tiling $T+t$ produces $T'+t$, which equals $T'$
by assumption. The uniqueness implies that  $T+t=T$.
Similarly, $T$ can by obtained by the substitution
from another tiling from the family,
which has also period $t$. In this 
way we can obtain tilings of the plane with tiles of arbitrarily large
size whose period is $t$, which is obviously impossible.}.
Hence all A2 tilings are non-periodic.

In this paper, we focus on the first A2 family
from \cite[Chapter 10.4]{GS} defined by the following 
\emph{Arrow rule}\label{lr1} (see \cite[Fig. 10.4.1(a)]{GS}):
\begin{quote}
{\em Color in a given 
tiling 
the sides of large and small tiles, 
as shown in Fig.~\ref{orient}(b,c).
The Arrow rule requires that for every pair of adjacent tiles 
each arrowed edge must fit against an edge with the same 
color pointing in the same direction.}
\end{quote}
We will call 
tilings that satisfy this rule 
\emph{A2  tilings}.
\begin{figure}
        \begin{center}
\includegraphics[scale=0.7]{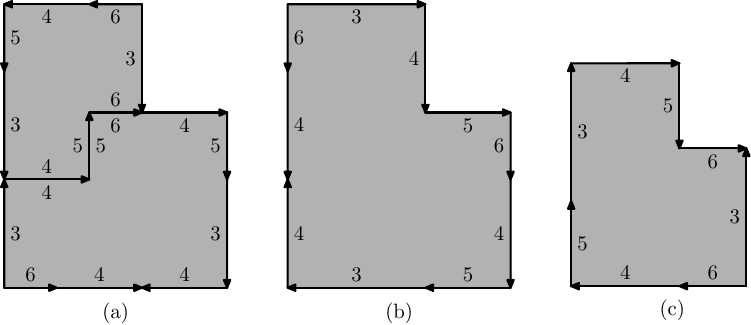}
        \end{center}
\caption{\small \small The Arrow rule for A2  tilings.
The sides of tiles in this figure
are divided into segments labeled by digits with arrows.
Digits represent the colors and arrows identify  
orientations of segments.
Digits correspond to the lengths 
of segments ($i$ means the length proportional to $\psi^i$).
Each arrowed edge must fit against an edge with the same 
label pointing in the same direction, e.g., as in (a).}\label{orient}
\end{figure}
Our main result describes the structure 
of A2 tilings of the plane (Theorem~\ref{description}) 
in the following terms.
A \emph{supertile} is a tiling that is obtained from a single 
large tile by applying to it $n$ substitutions
for some natural $n$, which is called the
\emph{level} of the supertile (see Fig.~\ref{f1}).
\begin{figure}[t]
        \begin{center}
          \includegraphics[scale=0.25]{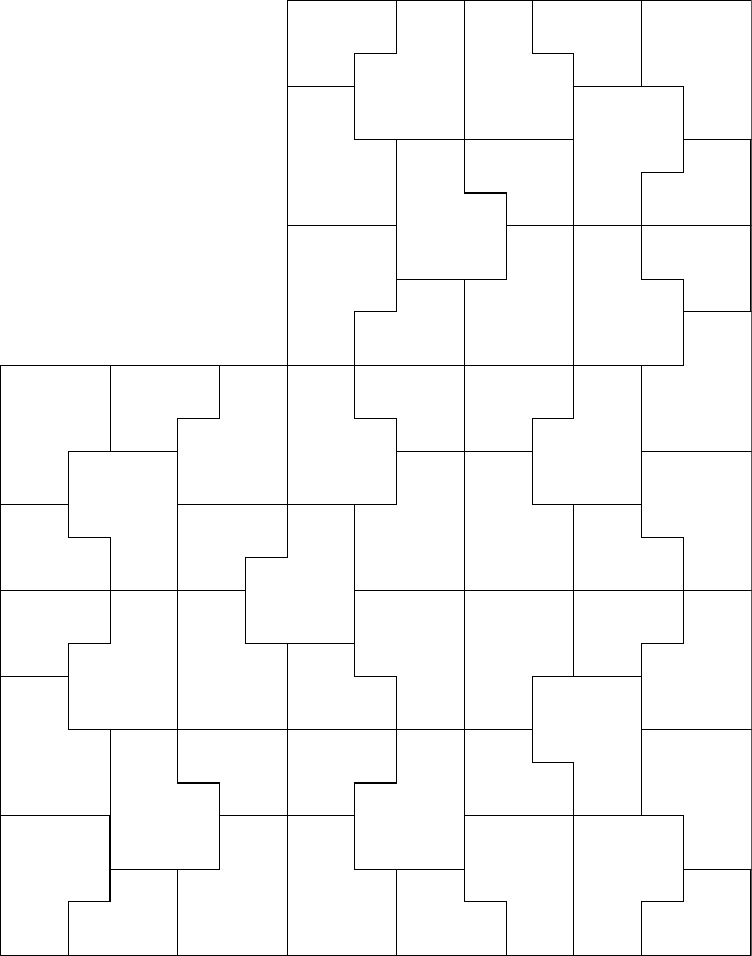}
        \end{center}
\caption{\small A supertile of level 8.
}\label{f1}
\end{figure}
Each supertile tiles the  
tile from which it was produced by substitutions.
An \emph{infinite supertile} is 
a union of an infinite chain of supertiles
$$
T_0\subset T_1\subset T_2\subset\dots 
$$
such that for all $n$ the tiling
$T_n$ tiles either the son, or the daughter of the
tile tiled by $T_{n+1}$ (see Fig.~\ref{supertile}).
\begin{figure}[t]
        \begin{center}
\raisebox{1cm}{\includegraphics[scale=0.1]{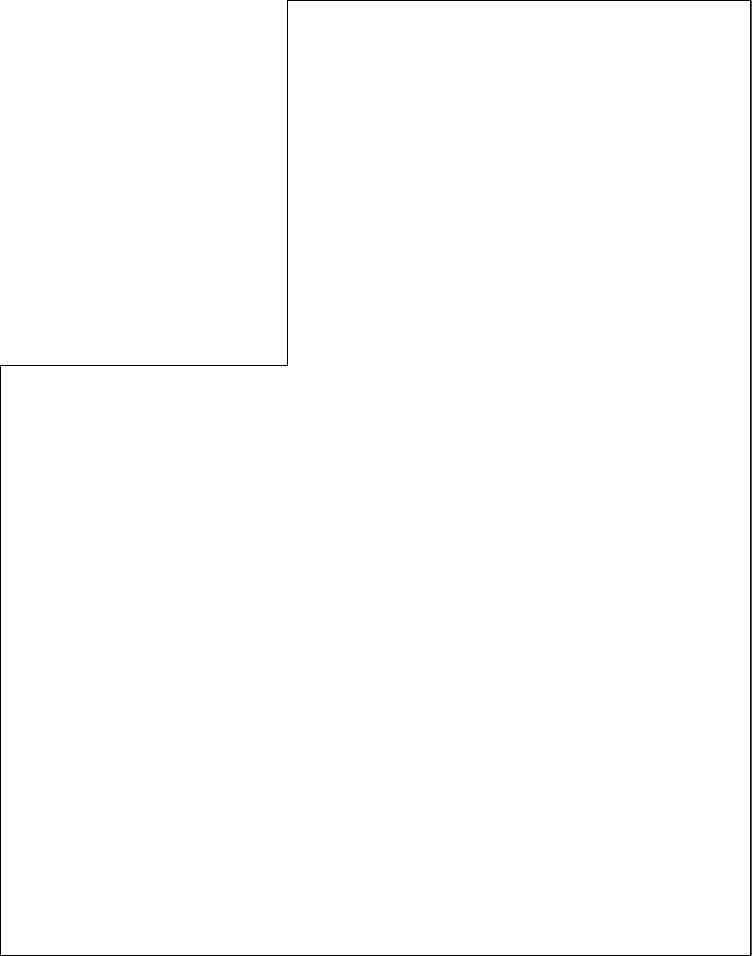}
$\longrightarrow$
\includegraphics[scale=0.1]{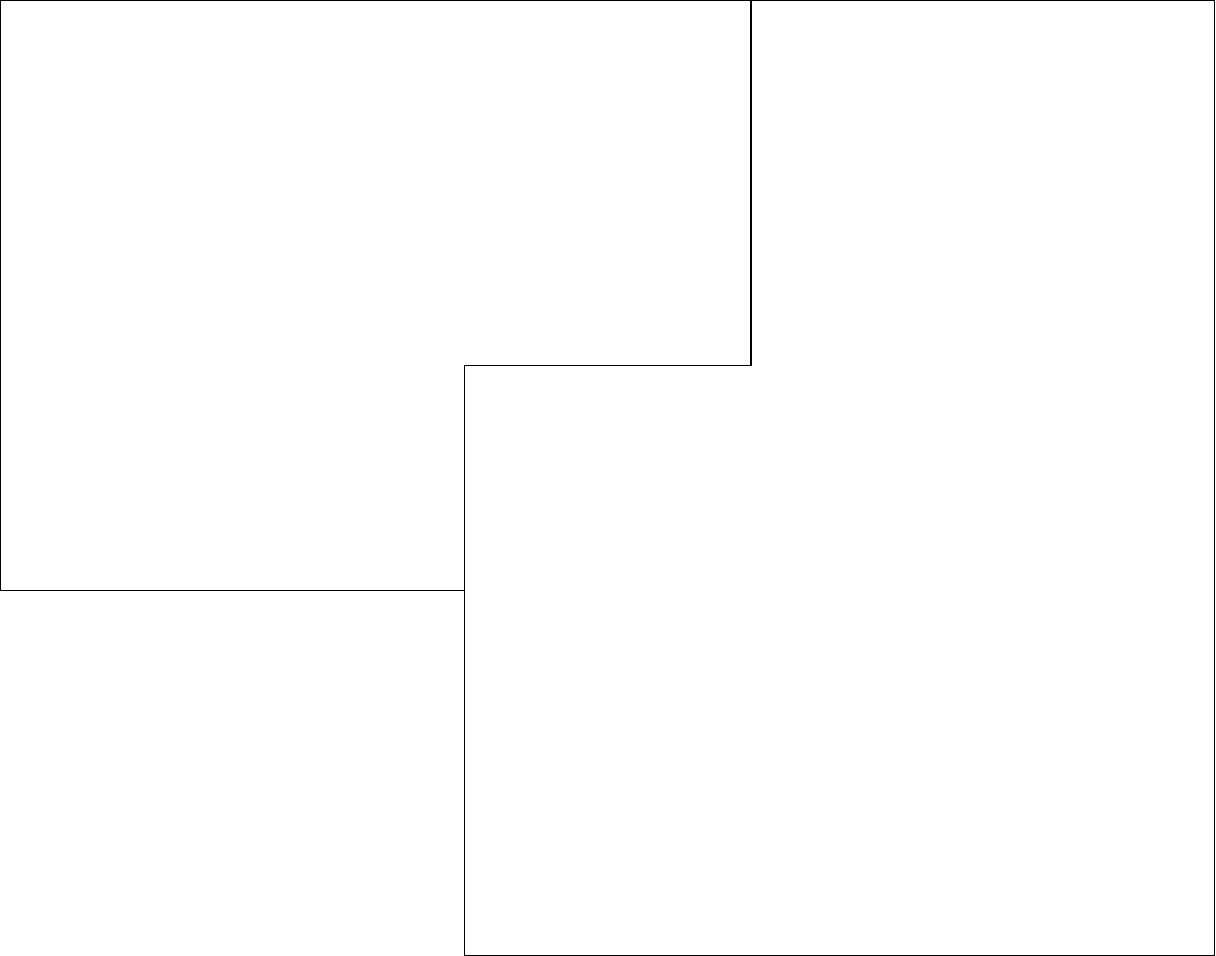}$\longrightarrow$}
\includegraphics[scale=0.1]{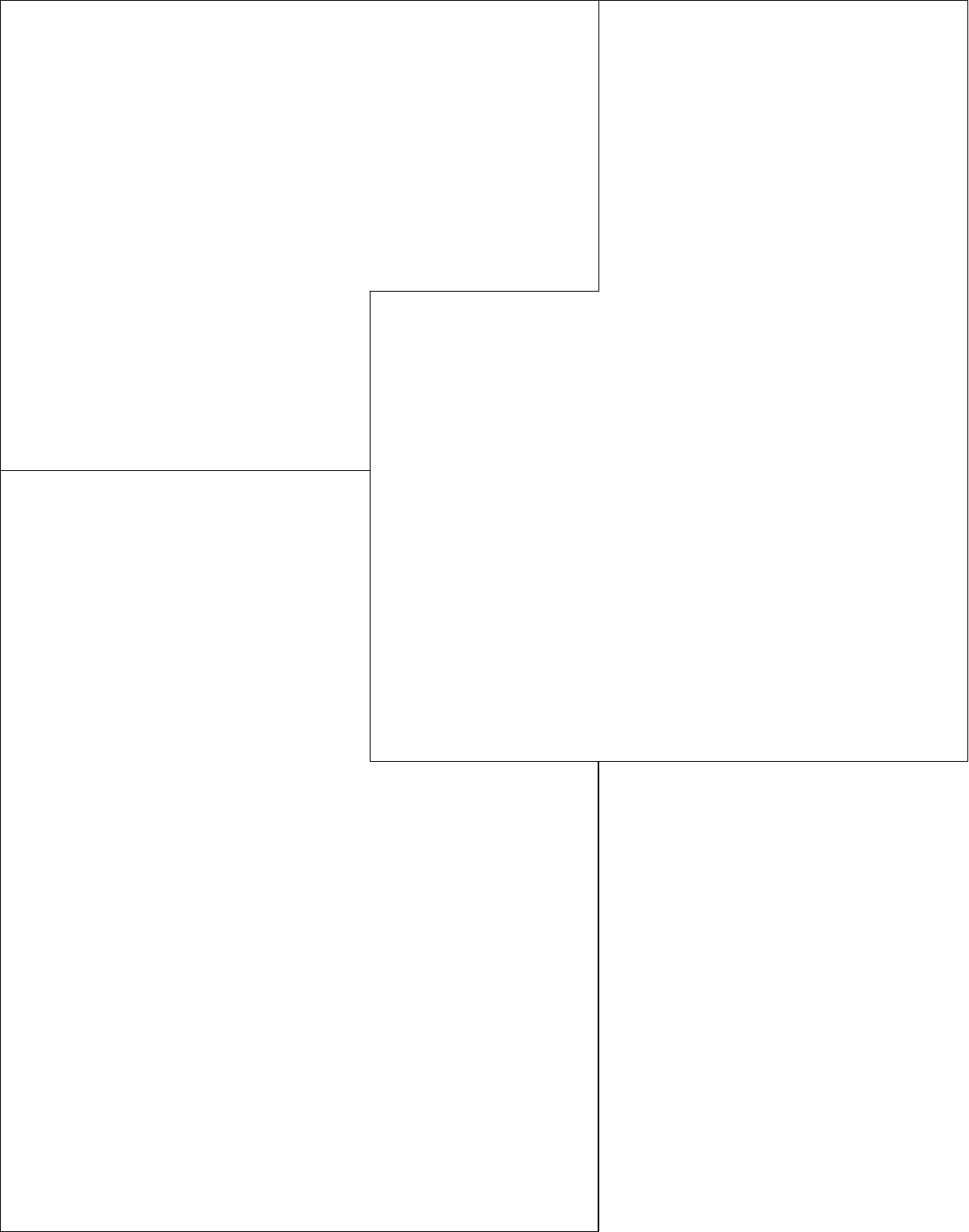}\raisebox{1cm}{$\longrightarrow$}
\includegraphics[scale=0.1]{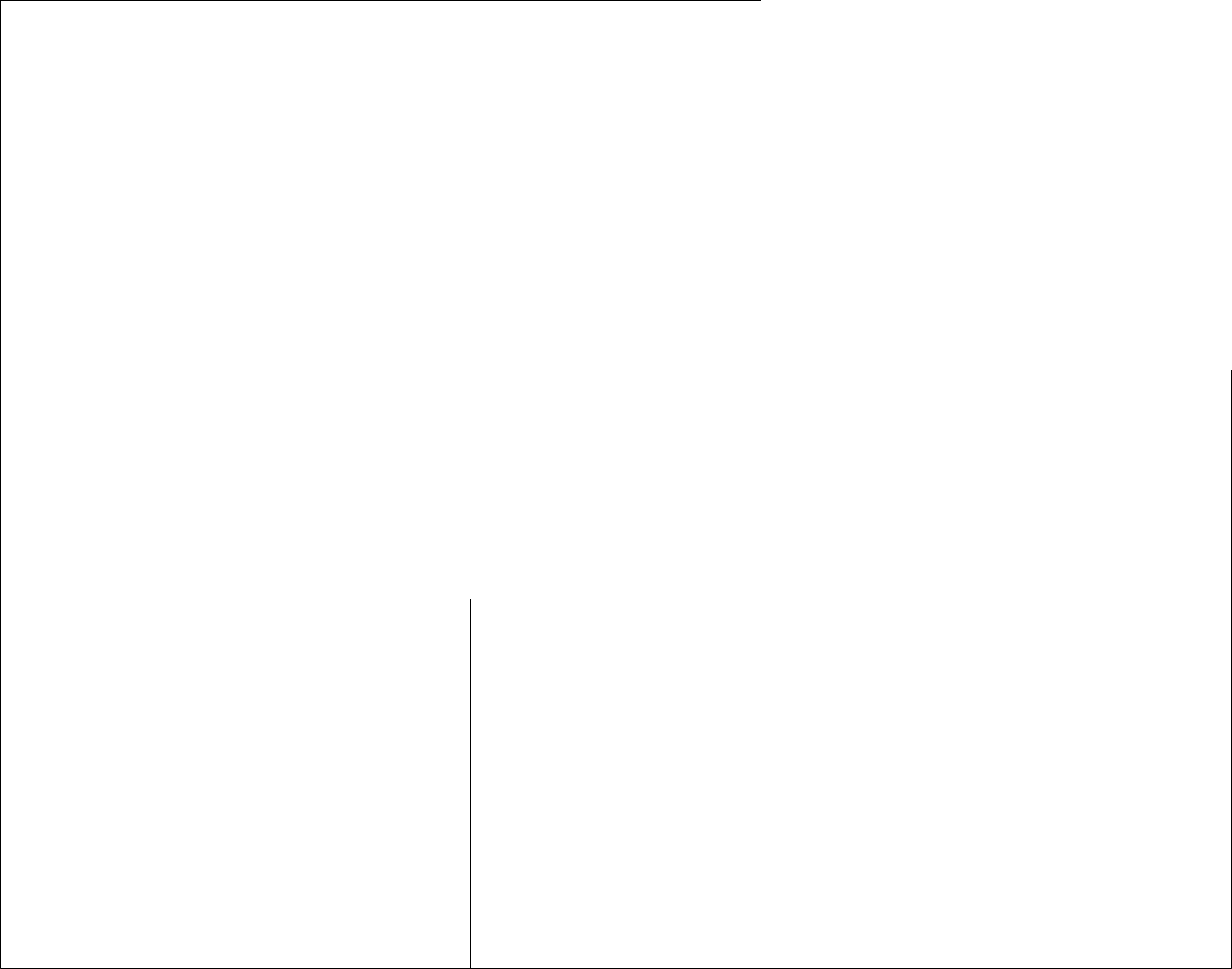}
        \end{center}
\caption{\small An infinite supertile is a union of a chain
of supertiles.}\label{supertile}
\end{figure}

Our main Theorem~\ref{description} states that every A2 tiling of
the plane is 
\begin{itemize}
\item either an infinite supertile, 
\item or a union of 2 infinite 
  supertiles $S_1,S_2$,  which both tile half-planes
  obtained by cutting the plane by a line $l$; moreover,
$S_1,S_2$ are reflections of each other in the axis $l$ (see Fig.~\ref{f75}), 
\item or a union of 4 infinite
supertiles $S_3,S_4,S_5,S_6$,  which all tile quadrants
obtained by cutting the plane by two orthogonal lines 
$l_1,l_2$; moreover,   
$S_3$ and $S_4$, as well as $S_5$ and $S_6$, are reflections
of each other in the axis $l_1$ and 
$S_3$ and $S_5$ ($S_4$ and $S_6$) are reflections
of each other in the axis $l_2$ (see Fig.~\ref{f75}).
\end{itemize}
\begin{figure}[t]
\begin{center}
\includegraphics[scale=.8]{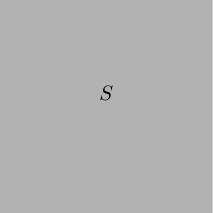} 
\hskip 1cm \includegraphics[scale=.8]{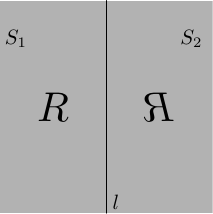}
\hskip 1cm \includegraphics[scale=.8]{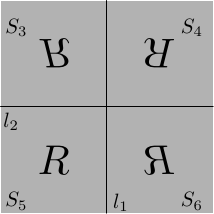}
        \end{center}
 \caption{\small Three different types of A2 tilings of the plane.
Here $S,S_1,S_2,S_3,S_4,S_5,S_6$ are 
infinite supertiles. The supertiles $S_1$ and $S_2$ are reflections
of each other in the axis $l$. 
The supertiles $S_3$ and $S_4$ ($S_5$ and $S_6$) are reflections
of each other in the axis $l_1$.
The supertiles $S_3$ and $S_5$ ($S_4$ and $S_6$) are reflections
of each other in the axis $l_2$.}\label{f75}
\end{figure}

Following~\cite S, 
we then consider the family of self-similar 
tilings\footnote{\emph{substitution tilings} in the terminology of~\cite{G}}
associated with our substitution. 
A tiling $T$ is called \emph{self-similar} 
if any its finite pattern can be found in a supertile.
It is not hard to show by induction that every 
supertile satisfies the Arrow rule
and hence every self-similar tiling
is an A2 tiling.
Our second result states that the converse 
implication holds as well: every A2 tiling of the plane is self-similar
(Theorem~\ref{th34}). This result follows from our first result  
on the structure of A2 tilings of the plane.

Finally, we answer the following question
about ``patching holes'' in A2 tilings.
Assume that a $d$-tiling $T$ of the plane satisfies
the Arrow rule everywhere except for a
bounded region; is there an A2 $d$-tiling
$T'$ of the plane such that the symmetric difference
of $T$ and $T'$ is finite?
We show in Theorem~\ref{th-ray} that this is not the case.


The paper is organized as follows.
In the next section we provide the main definitions.
In Section~\ref{sr} we state our results.
In Section~\ref{sp} we prove all theorems.
The proofs of propositions and lemmas are deferred to Appendix.

\section{Definitions}
\label{sd}

The notation $A\sqcup B$ refers
to the disjoint union of $A$ and $B$ and 
$A\subset B$ means that $A$ is a subset of $B$
(not necessarily a proper subset).

\begin{definition}
A \emph{tile of size $d$} 
is a Golden Bee of size $d$.
We call tiles of size $d$ \emph{large $d$-tiles}
and tiles of size $\psi d$ \emph{small $d$-tiles}.
A \emph{tile} is a $d$-tile for some $d$. 
If a $d$-tile $H$ is cut into small and large $d\psi$-tiles, $F$ and
$G$,
as shown below,
      \begin{center}
\includegraphics[scale=0.7]{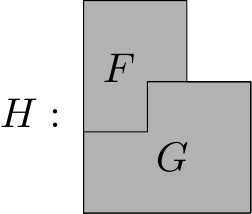}
        \end{center}
then we call $G$  and $F$
the  \emph{son} and the \emph{daughter}
of $H$, respectively,
we call $F$ a \emph{sister} of $G$
and call $G$ a
\emph{brother} of $F$.
The form of Golden Bees ensures  that
each tile has the unique sister
and the unique brother.\footnote{Indeed, if a small tile $F$ and
  a large tile $G$ are located as shown on the picture,
  then $G$ can be identified by $F$, as 
the unique large tile whose angle formed by sides
of length $\psi^3$ and $\psi^6$ fills 
the cavity of $F$ in such a way that the side
of length $\psi^5$ is shared by the side of small tile
of the same length. In a similar way the tile $F$ can be identified by $G$.}
\end{definition}

\begin{definition}
\emph{A $d$-tiling} is 
a non-empty set consisting of $d$-tiles  
that are pairwise disjoint (i.e., have no common interior 
points). A \emph{tiling} is a $d$-tiling for some $d$.
We denote by $[T]$ the set tiled by a tiling $T$.
\end{definition}


\begin{definition}\label{d1}
  The operation of \emph{substitution} $\sigma$ applied to a
  $d$-tiling $T$ 
produces a $d\psi$-tiling that is obtained from $T$ by 
cutting each large $d$-tile $A\in T$ into two tiles
of sizes $d\psi$ and $d\psi^2$,
as shown on Fig.~\ref{f2} (page~\pageref{f2}), and keeping all small
$d$-tiles intact.\footnote{It is more common 
in the literature to inflate
the initial tiling by  $1/\psi$ before substitution
so that the resulting tiling is again a $d$-tiling.}
The resulting tiling is denoted by $\sigma T$
and is called the \emph{decomposition} of $T$.
\end{definition}

Since 
each tile has the unique brother, substitution is an injective
operation. Indeed, if $\sigma T=T'$ for a $d$-tiling $T$,
then $T$ must consist of all $d$-tiles of the form
$(F\cup\text{the brother of }F)$,
where $F$ is a small $d\psi$-tile from $T'$, and 
of all large $d\psi$-tiles $G\in T'$ whose
sister is not in $T'$.

\begin{definition}
The inverse  operation $\sigma^{-1}$ is called  \emph{composition}.
\end{definition}

This operation is not total, that is, some tilings have no compositions.
For instance, if a $d\psi$-tiling $T'$ consists of a single small  
$d\psi$-tile, then  there is no $d$-tiling $T$
with $\sigma T=T'$.  
\begin{definition}
If $\sigma^{-1}(T)$ is defined, we say that $T$ 
is \emph{composable}.
If $\sigma^{-n}(T)$ is defined for all natural numbers $n$, 
we say that $T$ 
is \emph{infinitely composable}.
\end{definition}

\begin{definition}
A \emph{$d$-supertile of level $n\ge0$} is the $d$-tiling 
$\sigma^n(\{H\})$ obtained by applying
$n$ substitutions to the initial $d/\psi^n$-tiling $\{H\}$
consisting of the single large $d/\psi^n$-tile $H$.
(A supertile of level 8 is shown
in Fig.~\ref{f1} on page~\pageref{f1}.)
We will use also the notation $\sss_d(H)$ for $\sigma^n(\{H\})$
to indicate the size of tiles in $\sigma^n(\{H\})$.
A \emph{$d$-supertile of level $-1$} is the $d$-tiling 
consisting of the single small $d$-tile.
\end{definition}

It follows from the definition that every
supertile of level $n\ge0$ is composable
and its  composition is
a supertile of level $n-1$.
Every supertile of level $n\ge 1$ is a disjoint
union of a 
supertile of level $n-1$ and a supertile of level $n-2$, see Fig.~\ref{f39}.
\begin{figure}[t]
        \begin{center}
\includegraphics[scale=0.25]{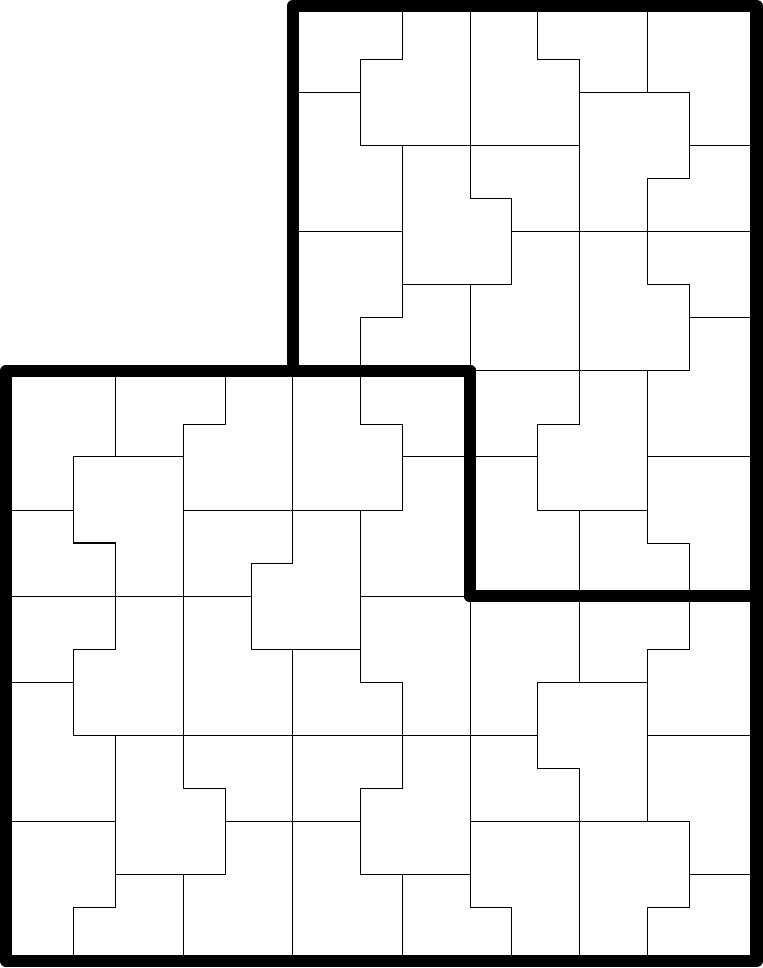}
        \end{center}
\caption{\small The supertile of level 8 
is represented as a union of supertiles of levels 7 and 6.}\label{f39}
\end{figure}

\begin{definition}
A $d$-tiling is called an \emph{infinite $d$-supertile}
if it is a union (= the limit) of an infinite chain of 
$d$-supertiles $T_0\subset T_1\subset T_2\subset \dots$
such that for all $n$ the tile
$[T_n]$ is either the son, or the daughter of the
tile $[T_{n+1}]$ (see Fig.~\ref{supertile} on page~\pageref{supertile}.)
\end{definition}

It is not hard to see that every infinite supertile is infinitely
composable.

\begin{definition}
Tilings $T$ and $S$ are called \emph{congruent}
if there is an isometry $f$ of the plane such that
$T=\{f(H)\mid H\in S\}$.
\end{definition}


\section{Results}\label{sr}

Our goal is two-fold: we want to understand how
A2 tilings of the plane may look like and, using that understanding,
to prove some their properties. 
It turns out that our technique works for tilings of 
any convex set, therefore we state our theorem for tiling of arbitrary 
convex sets (actually, we will see that, among convex sets, 
A2 tilings can tile only a plane, a half-plane or a  
quadrant).

The next proposition establishes some relations between
the notions of a supertile, an A2 tiling and an
infinitely composable tiling.

\begin{proposition}[\cite{GS,ammann}]\label{standard-correct}
(a) Every (finite or infinite) supertile is an A2 tiling.
(b) Each A2  tiling of a convex set is composable.
(c) The composition of every  A2  tiling 
of a convex set is again an A2 tiling (hence every
A2  tiling of a convex set is infinitely composable). 
\end{proposition}
 
For the sake of completeness we present a proof
of this proposition in Appendix.


\subsection{The structure 
  of A2 tilings of convex sets}

The structure 
of A2 tilings of convex sets
is established in Theorems~\ref{description1}
and~\ref{description} below.
The first theorem
applies to all  infinitely composable
tilings. The second one applies only to A2 tilings. 
Both theorems express possible structures of tilings in terms 
of infinite supertiles. Thus it is useful to
understand how infinite supertiles may look
like. Therefore we start with a
description of infinite supertiles.

Recall that an infinite $d$-supertile is a union of a chain
of $d$-supertiles 
$$
T_0\subset T_1\subset T_2\subset \dots.
$$
such that $[T_n]$ is either the son, or the daughter 
of $[T_{n+1}]$ for all $n$.
W.l.o.g. we may assume that  the supertile $T_0$ consists of a
single tile. 
In this case we will call the 
sequence of tiles 
$[T_0],[T_1],[T_2],\dots$
a \emph{representation}
of the infinite $d$-supertile $\bigcup_{n=0}^\infty T_n$.
This definition can be applied to finite supertiles as well,
in which case the sequence is finite.
It is not hard to see that for every 
sequence of tiles $H_0,H_1,H_2, \dots$ such 
that $H_n$ is the son or the daughter of  $H_{n+1}$ (for all $n$)
there is a unique infinite supertile with representation
 $H_0,H_1,H_2, \dots$.

A supertile can have many representations.
More specifically the following proposition holds.
\begin{proposition}\label{l8}
(a) Assume that $T$ is 
an infinite
$d$-supertile and $H$ is any its tile. 
Then there is a unique 
representation $H_0,H_1,H_2,\dots$ of $T$
with $H_0=H$. 
(b) For any two representations
$H_0,H_1,H_2,\dots$ and 
$G_0,G_1,G_2,\dots$ of an infinite supertile
$T$ there are $n,m$
such that  
$H_{i+n}=G_{i+m}$ for all $i\ge0$
(the representations have common tail).
(c)
If a tiling $T$ is infinitely composable 
and $H$ is any its tile, then there is a unique infinite supertile
$S$ with $H\in S\subset T$.
\end{proposition}

\begin{corollary}\label{c2}
  Any infinite supertile $T$ has only trivial
  symmetry (if $f$ is an isometry such that $f(T)=T$, then
  $f$ is the identity mapping).
\end{corollary}  
\begin{proof}
  Let $H$ be any tile from $T$ and let  $H_0,H_1,H_2,\dots$ 
be the unique 
representation of $T$
with $H_0=H$. Then $f(H_0),f(H_1),f(H_2),\dots$
is a representation of $f(T)=T$.
By Proposition~\ref{l8}(b) we have
$H_n=f(H_m)$ for some $m,n$. As $f$ does not change
the size of tiles, we must have $m=n$ and hence $H_n=f(H_n)$.
Since the Golden Bee has only trivial symmetry,
$f$ is the identity mapping.
\end{proof}  


Every representation $H_0,H_1,H_2,\dots$
of an infinite $d$-supertile $T$ is  completely specified 
by
the initial tile $H_0$ and 
the infinite sequence $\alpha$ of letters $s,l$
where $\alpha_n=s$ if $H_n$ is the daughter 
of $H_{n+1}$ and  $\alpha_n=l$ if $H_n$ is the son  
of $H_{n+1}$ ($s,l$ stand for ``small'' and ``large'').
The pair ($H_0$, $\alpha$) will be called a \emph{succinct
representation} of the infinite $d$-supertile $T$.
Now we formulate a simple criterion
of whether two 
infinite supertiles are congruent.
\begin{definition}
Define \emph{the weighted length} of a sequence $u$ of letters $s,l$  
by the formula: 
$$
w(u)= 2\text{(the number of $s$'s in }u)+ \text{(the number of $l$'s in }u).
$$ 
Infinite sequences $\aaa$ and $\beta$ of letters $s,l$
are \emph{equivalent} 
iff $\aaa$ and $\beta$ can be represented
as concatenations
$\aaa=u\gamma$ and $\beta=v\gamma$ for  some finite sequences
$u,v$
with   $w(u)=w(v)$. 
\end{definition}

The weighted length has the
following meaning: if $(H, u)$
is a succinct representation of a finite $d$-supertile
$S$, then the level of $S$ is equal to $w(u)$, if
$H$ is a large $d$-tile, and to $w(u)-1$
otherwise (see Fig.~\ref{f44}). 
\begin{figure}[t]
        \begin{center}
\includegraphics[scale=0.3]{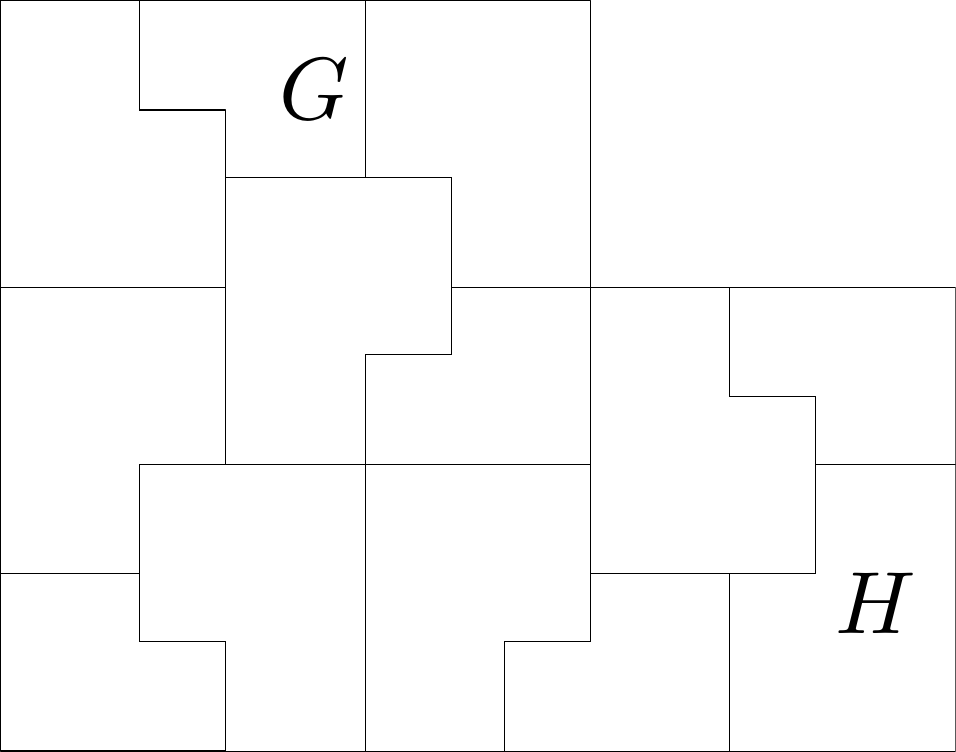}
        \end{center}
\caption{\small The supertile of level 5
has succinct representations
$(H, sls)$
and  $(G, ssll)$, where 
$w(sls)=2+1+2=5$ and $w(ssll)=2+2+1+1=6$.}\label{f44}
\end{figure}

\begin{proposition}\label{th-id}
Assume that $H$ and $G$ are large $d$-tiles.
Infinite $d$-supertiles with succinct representations $(H,\alpha)$  and 
$(G,\beta)$
are congruent iff $\aaa$ and $\beta$ are equivalent.
\end{proposition}

This proposition can be easily generalized to the case when
$H$ and $G$ are small tiles, or a large tile and a small tile.

The next theorem explains which part of the plane tiles an
infinite supertile with succinct representation $(H,\alpha)$:
\begin{theorem}\label{standard}
(a) An infinite supertile with succinct representation $(H,\alpha)$  
does not tile the entire plane iff 
a tail of $\aaa$ consists of the blocks
$s$ and $lsl$. (b) In this case (when $S$ does not tile the plane) 
it tiles a half-plane or a quadrant; 
more specifically, it tiles a quadrant 
iff a tail of $\aaa$ consists of alternating 
blocks 
$s$--$lsl$--$s$--$lsl$--\ldots.
\end{theorem}


Now we present a description of  infinitely composable
tilings.

\begin{theorem}\label{description1}
  Every infinitely composable tiling
can be represented as a disjoint union
of up to four
infinite supertiles; such a representation is unique.
\end{theorem}

By Proposition~\ref{standard-correct}
every A2  tiling  of a convex set 
is infinitely composable. Therefore this theorem applies
also to arbitrary A2  tiling of convex sets.
Theorem~\ref{description1} and Theorem~\ref{standard}
imply that such a set may be either the entire plane,
or a half-plane, or a quadrant. 
Indeed, these are the only convex sets which are 
disjoint unions of quadrants, half-planes and planes.
Moreover,   A2  tiling  of convex sets
have the following important property: if such a tiling
consists of more than one supertiles than those
supertiles must be axial symmetrical.

\begin{theorem}\label{description}
If an \arrows\  tiling  tiles a convex set,  
then that set is either a plane, or a half-plane, or a quadrant.
\\
(a) Every \arrows\  tiling  of the entire plane 
is either a supertile,
or a disjoint union of  \arrows\  tilings of half-planes; 
in the second case those tilings
of half-planes are axial symmetrical in the line that  
separates the half-planes.\\
(b) Every \arrows\  tiling of
a half-plane is either an infinite supertile, or  
a disjoint union of  \arrows\  tilings
of quadrants; in the second case those tilings
of quadrants     
are axial symmetrical in the line that  separates the quadrants.\\
(c) For every $d$ there are three different $d$-tilings
of a quadrant (up to congruence) and they all are infinite supertiles.
\end{theorem}


\begin{remark}
By this theorem every \arrows\  tiling $T$ of the 
plane has one and only one of the forms shown on Fig.~\ref{f75}
on page~\pageref{f75}. That form can be 
determined by any succinct representation of any infinite supertile
$S\subset T$. 
Thus there is a natural 1-1 correspondence between
\arrows\  tilings of the plane (we identify here 
congruent tilings) and equivalence classes of
infinite $l$-$s$-sequences.
\end{remark}


\subsection{\arrows\ tilings = self-similar tilings}

Now we proceed to our second result, which shows that every 
\arrows\ tiling of a convex set is self-similar.

\begin{definition}
A \emph{pattern} is a finite tiling. A pattern is \emph{legal} 
if it is a subset of a supertile.
A tiling $T$ is called \emph{self-similar} (with respect 
to the substitution  shown on
Fig.~\ref{f58} on page~\pageref{f58}) if all its finite subsets are legal.  
\end{definition}

By Proposition~\ref{standard-correct} 
every supertile is an \arrows\ tiling. 
Hence every self-similar tiling is 
an \arrows\ tiling.
Thus we have the following inclusions for tilings of convex sets:
$T$ is a self-similar tiling $\Rightarrow$
$T$ is an \arrows\ tiling $\Rightarrow$
$T$ is infinitely composable.
The second implication is not invertible 
(see Example~\ref{ex1} on page~\pageref{ex1}).
Our second result states that the first implication is.

\begin{theorem}\label{th34}
  Every
  \arrows\ tiling of a convex set is self-similar.
\end{theorem}
This theorem is not straightforward, as one might think.
We derive this theorem from Theorem~\ref{description}. One could try
to prove it directly. We outline a sketch of such proof
and point out the problems we face. 
\begin{proof}[Sketch of proof]
  Let $T$ be a given \arrows\ tiling of a convex set.
  We want to show that
it is self-similar. By Proposition~\ref{standard-correct}
the tiling $T$ is infinitely  composable.
We first
represent $T$ as a disjoint union of infinite supertiles.
To this end we apply compositions to $T$ and consider 
for each $n$ the tiling  $\sigma^{-n}T$. 

Pick any  tile $H$ from $T$. For every $n$, the tile $H$
is covered by a tile  $H_n$  from $\sigma^{-n}T$. Consider the 
tiling $S_n$ that consists of all tiles from $T$ that 
are covered by $H_n$. Then $S_n$ is a supertile.
The supertiles $S_n$ form a chain
$S_0=\{H\}\subset S_1 \subset \dots\subset S_n\subset\dots $
and their union is an infinite supertile. 
As all supertiles $S_n$ are 
self-similar, so is their union $S=\bigcup_{n=0}^\infty S_n$. 
If it happens that $S=T$, then we are done. 

Otherwise, if $S\ne T$, 
starting from any tile $B$ from the difference $T\setminus S$,
we can find a new  infinite supertile 
$S'\subset T$ which contains $B$. 
It is not hard to see that $S'$ and $S$ are disjoint
(indeed, if they shared a tile $C$, then both supertiles
$S,S'$ could be constructed
starting from $C$ as well and hence $S$ and $S'$ would coincide). 
In this way we can represent the given tiling
in the form $T=S_1\sqcup S_2\sqcup \dots$, where $S_1,S_2,\dots$ 
are infinite supertiles.

Now we face the following problem. 
We have to show that every pattern $W\subset T$
is legal. It may happen that $W$ intersects different $S_i$'s
from the representation $T=S_1\sqcup S_2\sqcup\dots$. 
In this case self-similarity of $S_i$'s does not imply that 
$W$ is legal.
We solve this problem as follows. 

Up to now we have only used infinite composability 
of $T$. As we have said,
there is an infinitely composable tiling of the plane 
which is not self-similar 
(Example~\ref{ex1} on page~\pageref{ex1}). 
So we have to use our assumption that $T$ satisfies
the Arrow rule. 
Here Theorem~\ref{description} comes into play.
By that theorem
the tilings $S_i$ are mirror images of each other as on Fig.~\ref{f75}
(page~\pageref{f75}).
Moreover, it follows
from the proof of Theorem~\ref{description}
that in the case  $T=S_1\sqcup S_2$ the pattern $W$
is covered by  two symmetrical supertiles $A_1,A_2$
\begin{figure}[t]
\begin{center}
\includegraphics[scale=1]{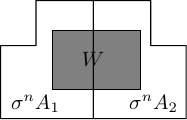} \hskip 2cm
\includegraphics[scale=1]{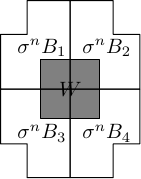}
        \end{center}
\caption{\small On the left: the pattern $W$
is covered by two symmetrical 
supertiles $A_1,A_2$. On the right:
the pattern $W$ is covered by four symmetrical
supertiles $B_1,B_2,B_3,B_4$.}\label{f25}
\end{figure}
and in the case  $T=S_1\sqcup S_2\sqcup S_3\sqcup S_4$
the pattern $W$ is covered by  
four symmetrical supertiles $B_1,B_2,B_3,B_4$,
as shown on Fig.~\ref{f25}.
It remains to note that both patterns
``two large tiles sharing their backs
so that they are reflections of each other
(as in  Fig.~\ref{f25} on the left)'' and
``four large tiles sharing their backs and bottoms
so that they are reflections of each other
(as in  Fig.~\ref{f25} on the right)''
are legal.
Indeed, they 
appear in the supertile of level 8 shown
on Fig.~\ref{f1} (page~\pageref{f1}).
Hence applying to the supertile of level 8 the appropriate 
number of substitutions we get a supertile that includes 
 $W$.
\end{proof}

\begin{remark}
As the family of self-similar tilings coincide 
with the family of \arrows\ tilings, Theorem~\ref{description}
applies to self-similar tilings as well. However, the direct proof
of Theorem~\ref{description} for self-similar tilings is only a little
bit easier than ours (namely, 
the proof of Proposition~\ref{standard-correct} is a bit simpler
for self-similar tilings).
\end{remark}

\subsection{Non-robustness of \arrows\ tilings}

Our third result states that \arrows\ tilings are sensitive to errors
in the following sense.
Durand, Romashchenko and Shen~\cite{brs}, who
considered tilings of the plane by square tiles,
defined the following notion of robust families of tilings.
Assume that we have a set $\tau$ of tiles,
where each tile is a square of size $1\times 1$
with colored edges. Consider the family $\T$ consisting
of tilings
of parts of plane in which tiles can be attached only
side-by-side so that the colors match.
The family $\T$ is called \emph{robust} if for any large enough
tiling of a set with a hole one can ``patch the hole'',
that is, one can find a tiling of the same set plus the hole
which differs from the original tiling not much.

\begin{definition}
Let $c_1<c_2$ be positive integers.
We say that a family of tilings $\T$ is \emph{$(c_1,c_2)$-robust}
if the following holds:  For
every positive natural $\Delta$ 
and for every tiling $T\in\T$ that tiles
a set
$$
S\supset ([-c_2\Delta,c_2\Delta]\times[-c_2\Delta,c_2\Delta])\setminus
([-\Delta,\Delta]\times[-\Delta,\Delta])
$$ 
there exists a tiling
$T'\in \T$ of the set $S\cup([-\Delta,\Delta]\times[-\Delta,\Delta])$ 
that contains all tiles from 
$T$ lying outside 
of the square $[-c_1\Delta,c_1\Delta]\times[-c_1\Delta,c_1\Delta]$.
\end{definition}

The smaller $c_1,c_2$ are the stronger this definition is.
Durand, Romashchenko and Shen~\cite{brs}
exhibited a family of non-periodic tilings
with many interesting properties
that is $(c_1,c_2)$-robust for some $c_1,c_2$.

The notion of a robust family naturally generalizes
to tilings by arbitrary tiles (of any shape)
defined by arbitrary local rules (like, say, the Arrow rule).
In this paper, we show that the family 
of \arrows\ tilings is 
\emph{not} $(c_1,c_2)$-robust for all $c_1,c_2$.
Moreover, the following is true:

\begin{theorem}\label{th-ray}
  There is a tiling $T$ of the
  plane that satisfies the Arrow rule everywhere except
  a bounded region and that has the following property:
  for any \arrows\ tiling $T'$ of the plane the difference
  $T\setminus T'$ is infinite.
\end{theorem}
\begin{corollary}
The family 
of \arrows\ tilings is 
\emph{not} $(c_1,c_2)$-robust for all $c_1,c_2$. 
\end{corollary}
\begin{proof}[Proof of the corollary]
  Let $T$ be the tiling from the theorem and $c_1,c_2$ arbitrary natural
  numbers.
Remove from $T$ all tiles violating the Arrow rule.
We obtain an \arrows\ tiling
of a set $S$, which is equal to the plane minus a bounded hole $H$.
Let $\Delta$ be equal to the diameter of the hole $H$ and hence
$S$ includes $[-c_2\Delta,c_2\Delta]^2\setminus
[-\Delta,\Delta]^2$. Assume now that an \arrows\ tiling $T'$
tiles the set $S\cup [-\Delta,\Delta]^2$,
that is, the entire plane.
We have to show that  $T'$ does not contain a tile from 
$T$ lying outside 
of the square $[-c_1\Delta,c_1\Delta]^2$.
By Theorem~\ref{th-ray}
the difference $T\setminus T'$
is infinite and hence 
at least one its tile
lies  outside that square.
\end{proof}


\section{Proofs of theorems}\label{sp}

In this section we prove all theorems. The proofs of propositions
and lemmas
are deferred to Appendix.
Several times in the proofs, 
we will apply composition to a part $S$
of a tiling $T$ and conclude that  $\sigma^{-1}T$
includes  $\sigma^{-1}S$. In general we cannot make
such conclusion, as
the following example demonstrates.
Let $S=\{G\}$ and $T=\{F,G\}$,
where $F,G$ are the daughter and the son
of a large $d/\psi$-tile $H$:
      \begin{center}
\includegraphics[scale=0.7]{pic63.pdf}
        \end{center}
Then $\sigma^{-1}S=\{G\}$ and $\sigma^{-1}T=\{H\}$.
Thus  $S\subset T$ while $\sigma^{-1}S\not\subset \sigma^{-1}T$.
However, this may happens only when $S$ contains a large tile $G$
whose cavity is not covered by $[S]$.
This makes possible for $T$ to include the sister of $G$,
in which case $G$ produces different tiles in
$\sigma^{-1}S$
and in $\sigma^{-1}T$.

A composable tiling $S$ is called \emph{proper} if the
cavity of every large tile from $S$ is covered by $[S]$.
The following tilings are proper:
\begin{itemize}
  \item Every supertile of level $n>0$ is proper.
  \item More generally,
    every tiling of the form $\sigma^2 T$ is proper,
    where $T$ is any tiling. (Indeed,
every large tile
from $\sigma^2 T$ is either a small tile from $\sigma^1 T$, in which
case its cavity is covered by its brother from $\sigma^1 T$, or
it is the brother of a small tile from $\sigma^2 T$, which covers
its cavity.)
\item Every tiling of a convex set is proper.
\end{itemize}
  
For proper tiling we have the following
\begin{lemma}\label{l0}
  (a)  If a proper tiling $S$  is a subset of a
  composable tiling $T$, then
  $\sigma^{-1}S\subset \sigma^{-1}T$.
  (b) If $T,S$  are  proper tilings
  then $\sigma^{-1}(S\cup T)=\sigma^{-1}S\cup \sigma^{-1}T$.
\end{lemma}

For reader's convenience, 
the following diagram represents the dependencies in the proofs:\\
Lemma~\ref{l0} $\Rightarrow$
Proposition~\ref{l8} $\Rightarrow$ 
Proposition~\ref{th-id}.\\
Proposition~\ref{l8} and Theorem~\ref{standard} $\Rightarrow$ 
Theorem~\ref{description1}.\\
Proposition~\ref{standard-correct}, Proposition~\ref{th-id},
Theorem~\ref{standard}, and
Theorem~\ref{description1}
$\Rightarrow$ 
Theorem~\ref{description}.\\
Proposition~\ref{standard-correct}
and
Theorem~\ref{description}
$\Rightarrow$ 
Theorem~\ref{th34}.\\
Lemma~\ref{l0}
$\Rightarrow$ Theorem~\ref{th-ray}.

\subsection{The proof of Theorem~\ref{standard}} 

Let $S$  be an infinite supertile with
succinct representation $(H_0,\alpha)$.
Consider two transformations  $s,l$  
of tiles: $s(H)$ is the unique tile whose
daughter is $H$, and $l(H)$ is the unique tile  whose
son is $H$.
Define $H_{i+1}= s(H_i)$ if
the $i$th letter of $\alpha$ is $s$
and $H_{i+1}= l(H_i)$ otherwise.
Then $H_0,H_1,H_2,\dots$
is a representation of $S$.

(a) Assume that $S$ does not tile the entire plane.
The sides of
hexagons $H_n$ stretch in two directions. Call
those directions \emph{horizontal and vertical} directions.
\begin{claima}
If the part of the plane tiled by $S$ 
intersects a vertical line and intersects
a horizontal line, then it includes the common points of the lines. 
\end{claima}
\begin{proof} 
For some $i$ both lines intersect $H_i$. 
This does not imply yet that $H_i$ contains their
common point, as it might happen that it falls into the cavity of
$H_i$. However in this case it falls into $H_{i+1}$.
\end{proof} 

By this claim every proper subset of the plane tiled by
an infinite supertile does not intersect a vertical 
or a horizontal line, call that line $L$. Then $S$ lies 
in one of the two half-planes defined by the line $L$.

Consider the distance $\delta_i$ from the tile 
$H_i$ to the line $L$.
As $H_{i+1}$ covers $H_i$, the sequence $\{\delta_i\}$ is non-increasing. 
Moreover, as $i$ is incremented by 1, 
the distance either remains the same, or 
decreases by some positive constant $\eps$ (or more).
Hence starting from some $i$ the distance does not change:
there are $\delta$ and $k$ such that $\delta_i=\delta$ for all 
$i\ge k$.

Shift the line $L$ towards the set $[S]=\bigcup_{n=0}^\infty H_n$ at the distance 
$\delta$. Now $L$  
touches all tiles $H_i$ with $i\ge k$. 
W.l.o.g. we may assume that $L$ is a horizontal line
and that all $H_i$ lie above $L$.
For  all $i\ge k$ the tile $H_i$ has a point 
on the line $L$ and hence an entire side
of the tile $H_i$ lies  on the line $L$.

We will view Golden Bees as ``chairs''
having the top, the back, the front, and the bottom (see Fig.~\ref{f15}).
\begin{figure}[t]
        \begin{center}
\includegraphics[scale=0.7]{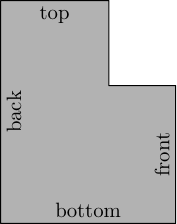}
        \end{center}
\caption{\small Then names of sides of a tile}\label{f15}
\end{figure}
The side of $H_i$ lying on the line $L$ can be either the front side,
or the back side, or the bottom side, or the top side (see Fig.~\ref{f16}).
\begin{figure}[t]
        \begin{center}\begin{center}
\includegraphics[scale=1.3]{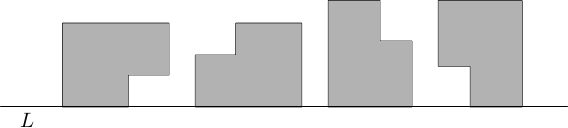}
        \end{center}
\end{center}
        \caption{\small Tiles lie
          on the horizontal line $L$
          on the front side, on the back side, on the bottom side
        and on the top side.}\label{f16}
\end{figure}
We need first to understand, how that side changes as $i$ increments. 
The following diagram shows how transformations $s$ and $l$
change the sides of tiles lying on the line $L$:
\begin{center}
\includegraphics[scale=1]{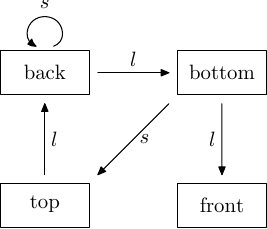}
        \end{center}
\label{diagram}
For example, the transition
$(\text{bottom}\stackrel{l}{\longrightarrow}\text{front})$
means that, if the line $L$ contains the bottom side of
a tile $H$, then
the front side of $l(H)$ lies on $L$.
        This fact is easy to verify by observing the cut in Fig.~\ref{f2}
        (page~\pageref{f2}).
If a transition is absent in this table,
then such case is impossible. For instance,
if $H$ lies on the line $L$
on its top side, then the tile $s(H)$ crosses the line $L$.

This diagram can be viewed as a finite automaton. That automaton
has the following 
property: in whatever state we start and whatever infinite sequence
of transitions we follow, we will always pass through the
``back'' state.  
Thus, for some  $i\ge k$ 
the back side of $H_i$ must  
lie on the line $L$. Moreover, there are infinitely many such 
$i$'s and between any two consecutive such $i$'s   
only the transition $s$ or the sequence of transitions 
$lsl$ may occur. This completes the `only if' part of
Theorem~\ref{standard}(a).

\begin{remark} \label{r1}
It follows from the above argument, that if an infinite supertile $S$
with representation $H_0,H_1,H_2,\dots$ 
does not tile the entire plane, then for infinitely many 
$n$ the back side of the tile
$H_n$ lies on  the border line
of the area tiled by $S$.
\end{remark}

Conversely, assume that a tail of $\aaa$ consists
of blocks $s$ and $lsl$. We have to show
that the infinite supertile $S$ with succinct representation
$(H_0,\alpha)$ does not tile the entire plane.
W.l.o.g. we may assume that $\aaa$ itself consists
of blocks $s$ and $lsl$. Consider the line passing through the 
back side of the tile $H_0$. Then all tiles $H_n$ lie in the same
half-plane as $H_0$ does. Hence the tiling $S$ tiles at most a
half-plane.

(b) Let $S$  be an infinite supertile with
succinct representation $(H_0,\alpha)$.
Assume that  
$\aaa=u\beta$ where $\beta$ consists
of the alternating blocks
$s$ and $lsl$. 
In other words, $\beta$ consists
of the alternating  
$s$ and $l$.

Let us show that $S$ tiles a quadrant. 
Assume first that $u$ is empty.
The mapping $sl$ transforms the small green (gray in the black and white
image) tile into a large tile that is inscribed
in the same quadrant.
        \begin{center}
\includegraphics[scale=0.5]{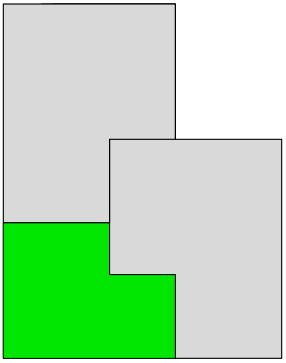}
        \end{center}
Therefore infinite number of applications of the transformation
$sl$ fills up the quadrant but not more. 
If $u$ is not empty, then 
the same arguments apply to some tile $H_n$ from the representation of $S$.

Assume now that a tail of $\aaa$ consists of the blocks
$s$ and $lsl$ but they \textsl{do not} alternate.
That is, the tail has infinitely many occurrences 
of  $ss$ or infinitely many occurrences 
of $lsllsl$.

Transformation $s$ maps a tile that lies 
on a line on its back to a larger tile 
that  also lies on the same line on its back.
The second application of $s$ increases the part  
of the tile that belongs to the line in the other direction.
        \begin{center}
\includegraphics[scale=0.5]{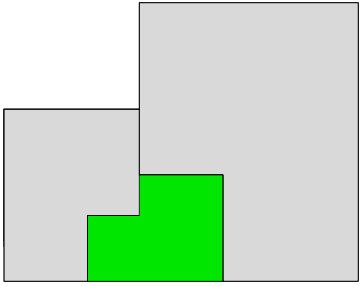}
        \end{center}
Thus if a tail of $\aaa$ has infinitely many occurrences 
of $ss$ (and consists of blocks $s$ and $lsl$)
then its application to the initial tile fills up
a half-plane.

Similar arguments apply  
when $\aaa$ has infinitely many of
occurrences of $lsllsl$. The mapping $lsl$
also maps a tile that 
is attached to a line by its back to a larger tile 
that is again attached to the same line by its back.
        \begin{center}
\includegraphics[scale=0.5]{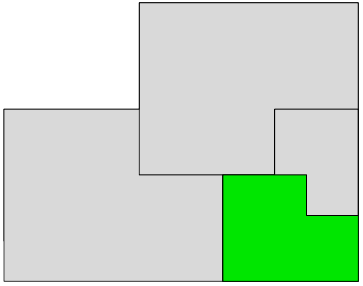}
        \end{center}
Thus the double application of $lsl$
increase the area of attachment in both directions.
Theorem~\ref{standard} is proved.

\begin{remark} \label{r2}
It follows from the above arguments, that if an infinite supertile
with representation $H_0,H_1,H_2,\dots$ tiles a quadrant, 
then for all large enough  
$n$ the back and bottom of the tile
$H_n$ lie on the boundary of that quadrant.
\end{remark}

\begin{example}\label{ex1}
Let $S$ be an infinite supertile tiling the half-plane
and let $S'$ be its reflection in the border line $l$ of the half-plane.
Shift $S'$ a little bit along $l$ (see Fig.~\ref{f53}). 
The resulting tiling is
an infinitely composable tiling of the plane and is not an \arrows\ tiling. 
\begin{figure}[t]        
\begin{center}
\includegraphics[scale=0.25,viewport=6cm 4.1cm 23cm 19.6cm,clip]{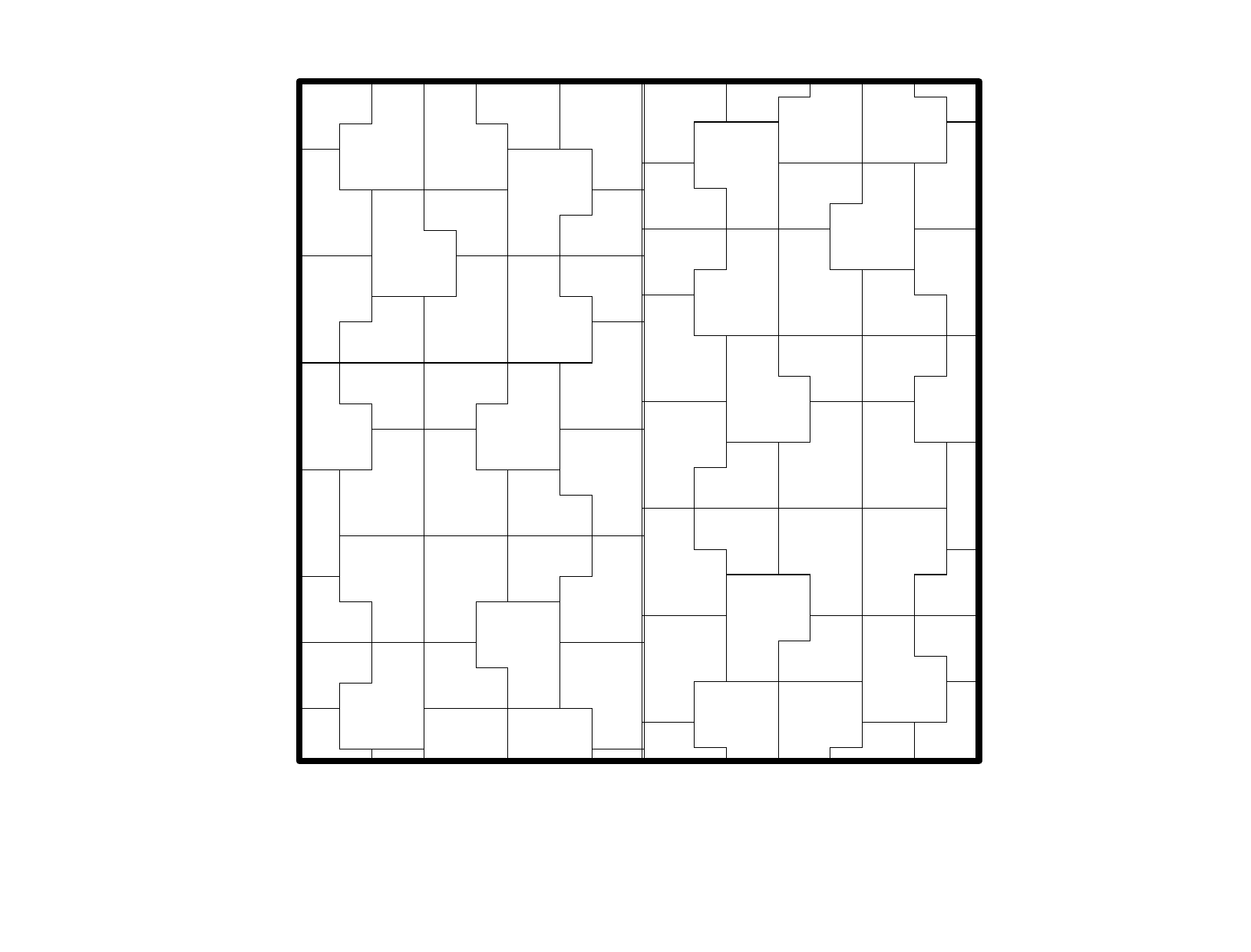}
        \end{center}
\caption{\small An infinitely composable tiling of the plane
  which is not \arrows.
}\label{f53}
\end{figure}

\end{example}

\subsection{The proof of Theorem~\ref{description1}}


Let $T$ be a given infinitely composable tiling.
We show first  
that $T$ can be represented in a unique 
way as a disjoint union of infinite supertiles.

Let $H$ be a large tile from $T$. By Proposition~\ref{l8}(c)
there is a unique infinite supertile
$S_1\subset T$ containing $H$.  
If $S_1$ coincides with $T$ then $T$ is an infinite 
supertile. Otherwise 
pick any tile $G$ in $T\setminus S_1$. Again by Proposition~\ref{l8}(c)
there is a unique infinite supertile
$S_2\subset T$ containing $G$. 
The supertiles $S_1,S_2$ are disjoint. Indeed,
assume that they share a tile $K$. 
Then we have both $K\in S_1\subset T$ and 
$K\in S_2\subset T$, which contradicts 
the uniqueness part of  Proposition~\ref{l8}(c)

As each infinite supertile covers at least a quadrant,  
in this way we can represent $T$  
as a disjoint union of up to four infinite  supertiles.  
Such representation is unique, as we have already shown that 
any two intersecting infinite supertiles that are subsets of $T$
coincide. 



\subsection{The proof of Theorem~\ref{description}}

Let $S$ be  an \arrows\  tiling  of a convex set.  
By Proposition~\ref{standard-correct} it is infinitely composable.
Thus by Theorem~\ref{description1}  
the tiling $S$ can be represented in a unique way
as a disjoint union up to four infinite supertiles,
and each of them tiles either the entire plane, or a half-plane,
or a quadrant. If a convex set is a  
disjoint unions of quadrants and half-planes, then 
it is either a plane, or a half-plane, or a quadrant.
This proves the first statement 
in  Theorem~\ref{description}.

\subsubsection{The proof of Theorem~\ref{description}(a)}

Let $S$ be  an \arrows\  tiling  of  the plane. 
As we have just seen, $S$ is either
a supertile, or a disjoint union of  \arrows\  
tilings $T,R$ of half-planes.
We have to show that in the second case   
tilings $T,R$     
are axial symmetrical in the line $l$ that  separates the half-planes.
Let  $R'$ be the reflection of $R$ in the axis $l$.
Then both  $T,R'$ tile the same half-plane and we have to show
that they coincide.

Call the set of colored oriented segments of tiles in $T$
lying on $l$  the \emph{shadow} 
of $T$  
\begin{center}
\includegraphics[scale=1]{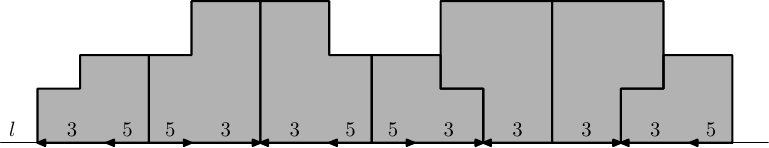}
 \end{center}
For instance, the sequence 
$$
\ltr\lfv\rfv\rtr\ltr\lfv\rfv\rtr \ltr\rtr \ltr\lfv
$$
is the shadow of the tiling from the above picture. 
Each segment in the shadow is identified by the triple
(start point, end point, color).
It suffices to prove 
the following
\begin{lemma}\label{l435}
Given the shadow of an \arrows\  tiling $T$ of 
a half-plane we can reconstruct the tiling.
\end{lemma} 
\noindent (Indeed, the shadows
of $T$ and $R'$ coincide, as $T\cup R$ is an \arrows\ tiling.
Thus by the lemma we have $T=R'$ and hence $R$ is the reflection of $T$.)

\begin{proof}[The proof of Lemma~\ref{l435}]
  Obviously
only
front, back, bottom and top sides of 
tiles can lie on border line of the half-plane tiled by $T$.
A quick look at the coloring of large and small tiles reveals
that these sides of tiles consists of 
blocks $\rsix\rfour, \rfive\rthree, \rthree,\rfour$
(see Fig.~\ref{f46}).
\begin{figure}[t]
\begin{center}
\includegraphics[scale=1]{pic15.pdf}\hskip 2cm
\includegraphics[scale=0.8]{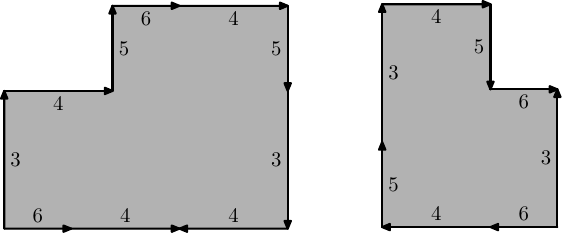}
       \end{center}
\caption{\small Coloring of large and small tiles.}\label{f46}
\end{figure}
Thus all shadows consist of these blocks.
Actually, blocks with odd numbers cannot occur
together with blocks with even numbers.

\begin{claimm}
  Every shadow either consists
  of 
  blocks $\rsix\rfour$, $\rfour$,
  or of blocks $\rfive\rthree, \rthree$.
\end{claimm}
\begin{proof}
Every tile has three sides 
that have colors $\rthree$ and $\rfive$ and
three sides 
that have colors $\rfour$ and $\rsix$. 
Call the sides of the first type \emph{odd} and
the sides of the second type \emph{even}. 
Every even side of a hexagon is parallel to every 
its even side and is orthogonal to every 
its odd side. 

In every \arrows\
tiling every two adjacent hexagons $G,H$ 
share a colored segment.
Thus they \emph{have the same orientation}:
odd sides of $G$ are parallel to odd sides of $H$
and are orthogonal to even sides of $H$.
This implies that all hexagons 
in an \arrows\ tiling of a convex set 
have the same orientation.
As all sides of the given tiling $T$
that lie on its border $l$ are parallel to each other,
either they all are even sides, or they all are odd sides. 
\end{proof}

Let us show first how to reconstruct from the shadow
all the tiles from $T$ 
that are adjacent to the border line $l$ of the half-plane.
Assume first that the shadow of  $T$ consists
of blocks $\rsix\rfour, \rfour$.

\begin{claim}\label{l22}
Given a shadow of $T$ consisting
of blocks $\rsix\rfour, \rfour$, we can 
reconstruct 
all the tiles from $T$ 
that are adjacent to the border line $l$ of the half-plane.
\end{claim}
\begin{proof}
  The given shadow can consist
  of fronts and backs of large tiles
  and tops and bottoms of small tiles
(see Fig.~\ref{f46}).
  However, if a large tile $H$ lies on
  a line on its front side, then
  both tiles $s(H)$, $l(H)$ cross that line
  (recall the diagram on page~\pageref{diagram}).
 Similarly, if a small tile $H$ 
  lies   on a line on its top side, then
  the tile $s(H)$ crosses that line.
 Hence  the  given shadow consists
  of bottoms of small tiles
and backs of large tiles. 
\begin{center}
\includegraphics[scale=1]{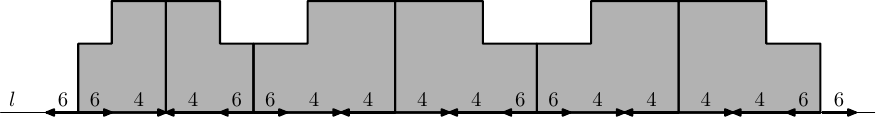}
        \end{center}
At the end of the bottom $\rsx\rfr$
of each small tile there is  
an orthogonal side (of the same tile) 
$\rfv\rtr\!\uparrow$
pointing \emph{to} the  interior of the half-plane (this is
easily verified by examining Fig.~\ref{f46}).
Only the back of another small tile can match that block $\rfv\rtr\!\uparrow$
and thus the bottom of each small tile 
$\rsx\rfr$ must be followed
by the symmetrical block $\lfr\lsx$. On the other hand, at the end
of the back side
$\rsx\rfr\lfr$ of every large tile there is a  
side  $\ltr\lfv\!\downarrow$ (of the same tile) pointing \emph{from} 
the  interior of the half-plane
(this is
easily verified by examining Fig.~\ref{f46}).  
Only the bottom of another large tile can match that 
block $\ltr\lfv\!\downarrow$ and thus the back $\rsx\rfr\lfr$ 
of every large 
tile must be followed by the block $\rfr\lfr\lsx$. 

This analysis shows that the shadow can be divided into blocks
$\rsx\rfr\lfr\lsx$ and $\rsx\rfr\lfr\rfr\lfr\lsx$. Such division
is unique, as the arrow on every digit 6 shows the direction to the 
block starting or ending by that digit.
We must attach to every block $\rsx\rfr\lfr\lsx$
a pair of small tiles lying on $l$ on their tops and sharing
their backs and to 
every block $\rsx\rfr\lfr\rfr\lfr\lsx$
a pair of large tiles lying on $l$ on their backs and sharing
their bottoms. 
\end{proof}

A similar lemma (with a similar proof) holds also 
for 3-5-shadows. However we do not need it, as we can finish the proof
as follows.

Using the   procedure of Claim 2,
we are able to reconstruct a given tiling $T$  from its shadow
in an arbitrarily large stripe 
along the border. Indeed, to
reconstruct $T$ in 
the stripe of width $d\psi^{2-i}$ near the border line,
first find the shadow
of the tiling $\sigma^{-i}T$ obtained from $T$ by $i$
compositions. Examining Fig.~\ref{orient} (page~\pageref{orient}),
it is not hard to verify that
the substitution transforms 
oriented colored segments according to the following rules 
\begin{align*}
            \rsix & \to \rfive \\
            \rfive & \to \rfour \\
            \rfour & \to \rthree \\
            \rthree & \to \lfour\lsix.
\end{align*}        
Applying to the given shadow the inverse 
map $i$ times (every $\lfour$ followed by $\lsix$
is replaced using the last line and remaining $\lfour$'s
are replaced using the second line), we are able to find
the shadow
of the tiling $\sigma^{-i}T$.
If it happens to be a 3-5-shadow then apply composition one
more time. 
Then by the procedure of Claim 2 we reconstruct the tiling $\sigma^{-i}T$
(or $\sigma^{-i-1}T$)
in the stripe of width $d\psi^{2-i}$ (or $d\psi^{2-i-1}$) 
near the border of the half-plane.
Finally, apply $i$ (or $i+1$) substitutions to the obtained tiling.

For instance, assume that we are given the shadow 
$$
\dots\lfour\rfour\lfour\lsix\rsix\rfour\lfour\rfour\lfour\lsix\rsix\rfour\lfour
\lsix\rsix\rfour\lfour\rfour\lfour\lsix\rsix\rfour\lfour\lsix\rsix\rfour\lfour\rfour
\dots
$$
and we want to reconstruct the tiling in the stripe of width $d$.
We apply the inverse map two times and get 
the shadow 
$$
\dots\lsix\rsix \rfour\lfour\lsix\rsix\rfour\lfour\rfour\lfour\lsix\rsix\rfour\lfour\rfour\lfour\lsix\rsix\dots.
$$
Then apply the procedure of Claim 2 to construct the tiling 
with this shadow: 
\begin{center}
\includegraphics[scale=1]{pic47.pdf}
        \end{center}
Then we apply substitution two times 
and get the sought tiling: 
\begin{center}
\includegraphics[scale=1]{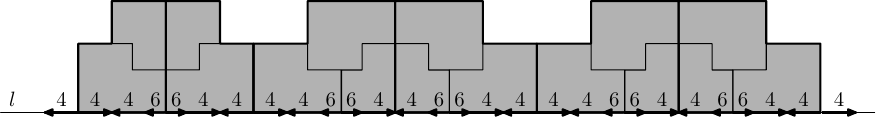}
        \end{center}
Lemma~\ref{l435} and Theorem~\ref{description}(a)
are proved.
\end{proof}

\subsubsection{The proof of Theorem~\ref{description}(b)}
\label{quadrant}
Let $S$ be an \arrows\  tiling of a half-plane,
which is not an infinite supertile. 
Let $l$ denote the border line of that half-plane. 
As we have seen, $S$ is
a disjoint union of two infinite supertiles $S_1,S_2$ tiling
quadrants. Let $r$ denote the ray that separates
those quadrants. Let $S_1',S_2'$ denote the reflections of $S_1,S_2$ 
in the axis $l$.

Then $S_1\cup S_2\cup S_1'\cup S_2'$ is an \arrows\  tiling of the entire plane,
which is a disjoint
union of tilings $S_1\cup S_1'$ and $S_2\cup S_2'$ of half-planes
separated by the line $r\cup r'$. By Theorem~\ref{description}(a)
the tilings
 $S_1\cup S_1'$ and $S_2\cup S_2'$ are reflections of each other in the
axis  $r\cup r'$, q.e.d.

\subsubsection{The proof of Theorem~\ref{description}(c).}
As we have seen, every \arrows\  tiling of a quadrant is an infinite supertile. 
Let us show that there are only three such tilings. 
Let $(H,\alpha)$ be a succinct representation 
of an infinite  supertile tiling a quadrant
where $H$ is a large tile.

By Proposition~\ref{th-id}
infinite supertiles with succinct representations 
$(H,\alpha)$ and $(G,\beta)$ are congruent iff 
$\aaa$ and $\beta$ are equivalent (we assume that both $H,G$ are large tiles). 
Recall 
that sequences  $\aaa,\beta$ of letters 
$l,s$ are equivalent if 
$\aaa=u\gamma$ and  $\beta=v\gamma$ for some 
$u,v$ of the same weighted length;
calculating the weighted length we count every letter 
$l$ with weight 1 and every letter $s$ with weight 2.

By Theorem~\ref{standard} the tiling with succinct representation $(H,\alpha)$ tiles a quadrant
iff $\aaa$ has a tail $slslslsl\dots$. 
Let us show that there are three
non-equivalent sequences $\aaa$ having such tail,
namely 
\begin{equation}\label{eq76}
slslslsl\dots, \qquad lslslslsl\dots,\qquad llslslslsl\dots
\end{equation}
Indeed, the weighted lengths of
the sequences $s$ and $ll$ coincide.
Thus replacing in any sequence any 
letter $s$ by $ll$ we get an equivalent sequence.
Vice verse, replacing any block $ll$ by $s$
we get an  equivalence sequence. 
Therefore every sequence with 
the tail $slslslsl\dots$
is equivalent to a sequence of the form
$uslslslsl\dots$ where 
$u$ is a finite sequence consisting only  of $l$'s, $u=ll\dots l$ (we replace
each $s$ before the tail by $ll$).  
Now replace in $u$ every triple of consecutive  
$l$'s by $sl$. The resulting sequence is equivalent 
to the original one and equals to 
one of the sequences~\eqref{eq76},
depending on the residue of the length of $u$ modulo 3.
 
On the other hand, as $w(sl)=3$
and $w(\text{empty word})$, $w(l)$, $w(ll)$
are not congruent modulo 3,
the three above sequences are pairwise non-equivalent. 
One can see in Fig.~\ref{f61}
\begin{figure}[t]
        \begin{center}
\includegraphics[scale=0.45]{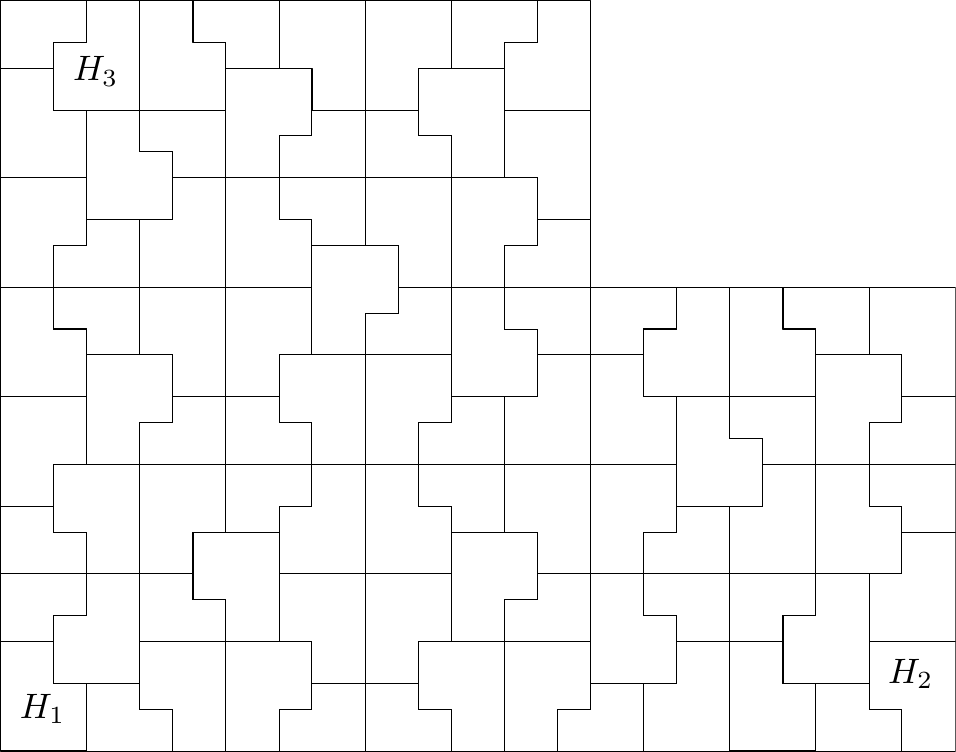}
        \end{center}
        \caption{\small The picture shows three different tilings
          of the quadrant: they have 
          succinct representation $(H_1,slslsl\dots)$,
          $(H_2,lslslsl\dots)$ and
           $(H_3,llslslsl\dots)$.}\label{f61}
\end{figure}
how the corresponding tilings 
of quadrants look like. 
The first one is obtained 
if we put the origin of the quadrant in the bottom left 
corner. 
To obtain the second tiling imagine that the origin of the quadrant
is in the bottom right corner. For the third
consider the top left corner. The substitution
transforms these tilings as follows: $1\to 2\to 3\to 1$.

\subsection{The proof of Theorem~\ref{th34}}

We are given an \arrows\ tiling $T$ of a convex set and have
to show that it is self-similar.
By Theorem~\ref{description}, if $T$
is not an infinite supertile (in which case we are done),
it consists either of two, or of four axial symmetrical infinite supertiles.

Consider the first case:   
$T=S_1\sqcup S_2$ where $S_1,S_2$ are infinite supertiles. 
Let $W$ be a finite  subset of $T$.
We have to show that $W$ is a subset of a supertile. 
Let $W_i=W\cap S_i$ for $i=1,2$. 
For an integer $n$, apply $n$ times composition
to tilings $S_1,S_2$.
If $n$ is large enough, then
the area tiled by $W_1$ is covered by a single tile $A_1$
from $\sigma^{-n}S_1$. By Remark~\ref{r1} on page~\pageref{r1}
w.l.o.g. we may assume that 
$A_1$ is a large tile and its back 
lies on the line $l$ separating $[S_1]$ from $[S_2]$.
The mirror image $A_2$  of $A_1$ belongs to $\sigma^{-n}S_2$. 
If $n$ is large enough then
the area tiled by $W_2$ is covered by $A_2$,
as shown on Fig.~\ref{f1a} (page~\pageref{f1a}).
\begin{figure}[t]
        \begin{center}
\includegraphics[scale=1]{pic25.pdf} \hskip 2cm
\includegraphics[scale=1]{pic24.pdf}\hskip 2cm
\includegraphics[scale=0.25]{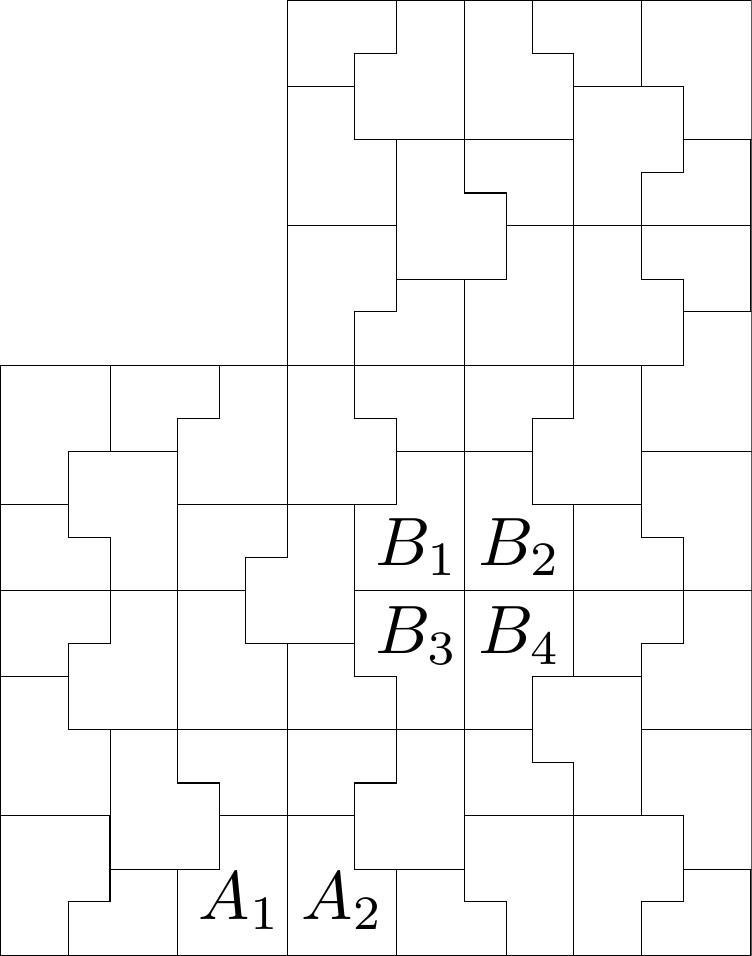}
        \end{center}
\caption{\small On the left: the pattern $W$
is covered by two symmetrical 
supertiles $\sigma^nA_1,\sigma^nA_2$. In the middle:
the pattern $W$ is covered by four symmetrical
supertiles $\sigma^nB_1,\sigma^nB_2,\sigma^nB_3,\sigma^nB_4$. On the right:
the supertile of level 8 includes
both patterns $\{A_1,A_2\}$ and $\{B_1,B_2,B_3,B_4\}$
}\label{f1a}
\end{figure}
It remains to notice the pattern consisting of two 
tiles sharing the their largest sides is included in the
supertile of level 8
(see Fig~\ref{f1a}).
Applying $n$ substitutions to that supertile 
we obtain a supertile including $W$.

Consider now the second case $T=S_1\sqcup S_2\sqcup S_3\sqcup S_4$
where $S_1,S_2,S_3,S_4$ are axial symmetrical 
infinite supertiles tiling  quadrants.
Let $W$ be a finite  subset of $T$
and let $W_i=T\cap S_i$  for $i=1,2,3,4$.
Arguing in a similar way we can show that 
for some $n$  
each set $[W_i]$ is covered by a single large tile  $B_i\in\sigma^{-n}T$. 
By Remark~\ref{r2} (page~\pageref{r2}) we may assume that 
$B_1,B_2,B_3,B_4$ are large tiles sharing the their backs and 
bottoms, as shown on Fig.~\ref{f1a}.
The pattern consisting of four large tiles shown in Fig.~\ref{f1a}
can be found in the supertile of level 8 (Fig.~\ref{f1a}).
Applying $n$ substitutions to that supertile 
we obtain a supertile including $W$.


\subsection{The proof of Theorem~\ref{th-ray}}

The tiling $T$ (a finite part of it)
is shown on Fig.~\ref{f48} (page~\pageref{f48}).
\begin{figure}[t]
        \begin{center}
\includegraphics[scale=0.7]{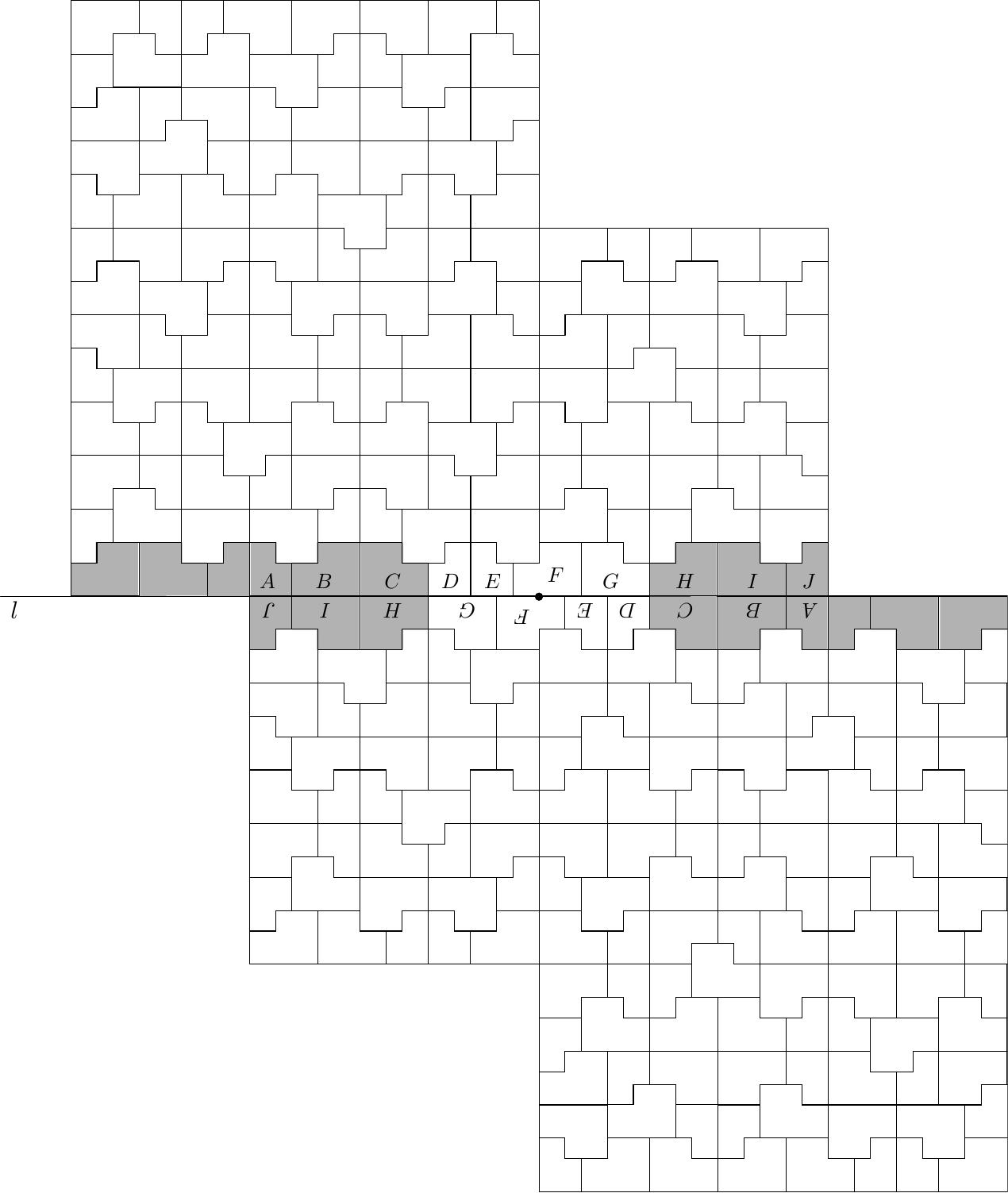}
        \end{center}
        \caption{\small The tiling $T$ is the union of
          tilings $T_{up}$ and $T_{bottom}$. The tiling  $T_{bottom}$ is
          obtained from  $T_{up}$ by rotation by $180^{\circ}$ around
          a point on the line $l$. Some tiles from $T$ have names. 
          }\label{f48}
\end{figure}
This tiling is defined as follows. 
Consider a large $d$-tile $F$ and the point $P$ on its
largest side (see Fig.~\ref{f51}
on page~\pageref{f51}). 
\begin{figure}[ht]
        \begin{center}
          \includegraphics[scale=0.1]{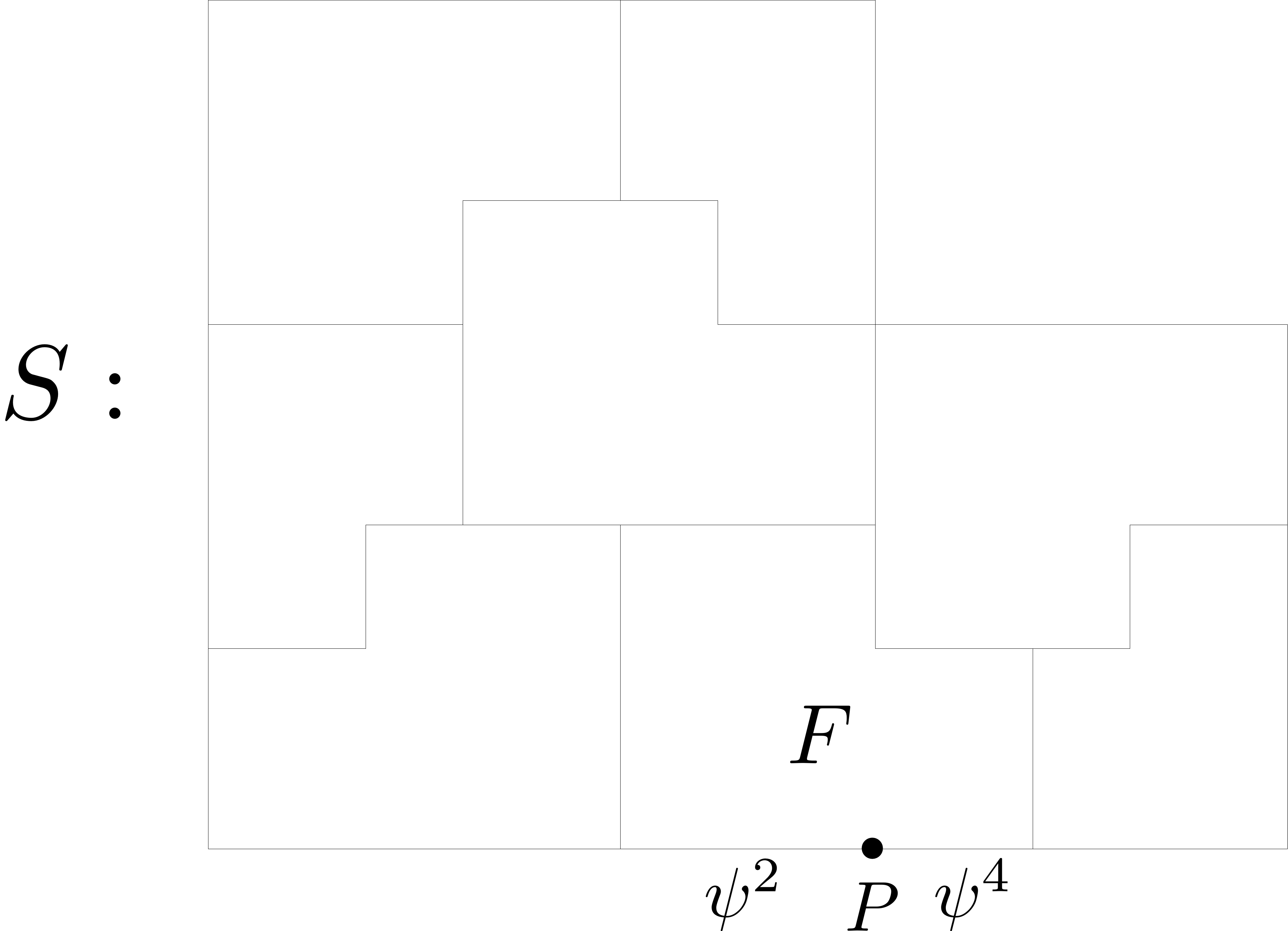}
        \end{center}
\caption{\small The picture shows 
  a $d$-supertile $S$ of level 4.
  It contains a large $d$-tile $F$.
  The supertile $S$ can be obtained from $F$
  by applying the homothety with center $P$
  and then applying four substitutions.
  The point $P$ divides the largest side of $F$ (of length
$d$) into segments of lengths  $d\psi^2,d\psi^4$.}\label{f51}
\end{figure}
Let $\HH$ denote the  homothety with center $P$ and ratio $\psi^{-4}$
and $\F$ the mapping $X\mapsto \sigma^4(\HH(X))$ on $d$-tilings.
Consider the supertile $\F(\{F\})$
(the supertile $S$ on Fig.~\ref{f51}).
It is easy to verify that it indeed contains the tile $F$.
This implies that
$$
\{F\}\subset \F(\{F\})\subset \F(\F(\{F\}))\subset \dots
$$
Let $T_{up}$ be the union of this chain of tilings:
$$
T_{up}= \{F\}\cup \F(\{F\})\cup \F(\F(\{F\}))\cup
\dots.
$$
One can see that
the tiling  $T_{up}$ is an infinite supertile with succinct representation
$(F,ssss\dots)$.

Now consider the rotation $\RR$ by $180^{\circ}$ around the point $P$
and let
$$
T_{bottom}=\RR(T_{up}) \text{ and } T=T_{up}\cup T_{bottom}.
$$
It is easy to see that both
$T_{up}$ and $T_{bottom}$ are fixed points of the mapping $\F$.

The theorem follows from two claims:
\begin{claimm}\label{claim1}
  The tiling $T$ satisfies the Arrow rule everywhere
  except for sides shared
  by pairs of tiles $(E,\RR(F))$, $(F,\RR(F))$ and $(F,\RR(E))$.
\end{claimm}
\begin{claim}\label{claim2}
  If $T'$ is an infinitely composable tiling of the plane that includes
  almost all tiles from $T$, then $T'=T$.
\end{claim}

\begin{proof}[Proof of Claim~\ref{claim1}]
  If two adjacent tiles from $T$ both belong to
  $T_{up}$ or both belong to $T_{bottom}$, then they
  satisfy the  Arrow rule, as both  $T_{up}$ and $T_{bottom}$
  are infinite supertiles. Therefore it remains to verify
  the Arrow rule is met for all tiles from
  $T\setminus\{F,\RR(F)\}$
  that are adjacent to the line $l$ separating $T_{up}$ from
  $T_{bottom}$ 
  (see Fig.~\ref{f48} on page~\pageref{f48}).
  For the pairs $(D,\RR(G))$, $(E,\RR(G))$, $(G,\RR(E))$, 
  $(G,\RR(D))$ this verification can be done by hand. 

  For the remaining tiles we can argue as follows.
  Let $V$ denote the set of all tiles from $T_{up}$ that are adjacent
  to $l$ and are not marked grey on Fig.~\ref{f48}, that is,
  $$
  V=\{D,E,F,G\}.
  $$
  Let $W$ denote the remaining tiles from $T_{up}$ that have
  names (they all are  marked grey), that is,
  $$
  W=\{A,B,C,H,I,J\}.
  $$
  The claim follows from the following two
  facts that can be verified by hand:\\
\emph{Fact 1.} $\F(V)=V\cup W \cup U$,
where $U$ is a set of tiles that all do not touch the line $l$
separating $T_{up}$ from  $T_{bottom}$
(we can find $U$ explicitly, however for
our argument we need only that all tiles from $U$
do not touch the line $l$).\\
\emph{Fact 2.}
Let $\SS$ denote the reflection in the axis $l$. Then
$\RR(W)=\SS(W)$. More specifically,
$\RR(A)=\SS(J)$,
$\RR(B)=\SS(I)$ and so on.

The definition of $T_{up}$ and 
the first fact imply that 
$$
T_{up}=V\cup\bigcup_{i=0}^{\infty} 
\F^i(W)\cup \bigcup_{i=0}^{\infty} \F^i(U).
$$
Indeed, it is easy to show by induction on $n$
that
$$
\F^n(\{F\})\subset V\cup\bigcup_{i=0}^{n} 
\F^i(W)\cup \bigcup_{i=0}^{n} \F^i(U)
$$
for all $n$.
Conversely, it is easy to see that $V\subset  \F^2(\{F\})$ hence
$V\cup W \cup U= \F(V)\subset \F^3(\{F\})$,
which implies that 
$$
\F^i(W)\cup \F^i(U) \subset \F^{i+3}(\{F\})\subset T_{up}
$$
for all $i$.

Similarly,
$$
T_{bottom}=\RR(V)\cup\bigcup_{i=0}^{\infty} 
\F^i(\RR(W))\cup \bigcup_{i=0}^{\infty} \F^i(\RR(U)).
$$
These representations of $T_{up}$ and $T_{bottom}$
and the second fact imply that all tiles from
$T_{up}\setminus V$ that are
adjacent to the line $l$ 
are  mirror images of tiles from $T_{bottom}\setminus \RR(V)$. Hence
the Arrow rule is met for those tiles.
\end{proof}

\begin{proof}[Proof of Claim~\ref{claim2}]
  Consider the chain of tilings
  $S_0\subset S_1\subset S_2\subset \dots$ where
$S_0=\{F,\RR(F)\}$ and $S_{i+1}=\F(S_i)$ (recall that $\F(S)=\sigma^4\HH(S)$).
The set $S_1=\F(S_0)$ is marked grey  
on Fig.~\ref{f62} (page~\pageref{f62}).
\begin{figure}[t]
        \begin{center}
\includegraphics[scale=0.7]{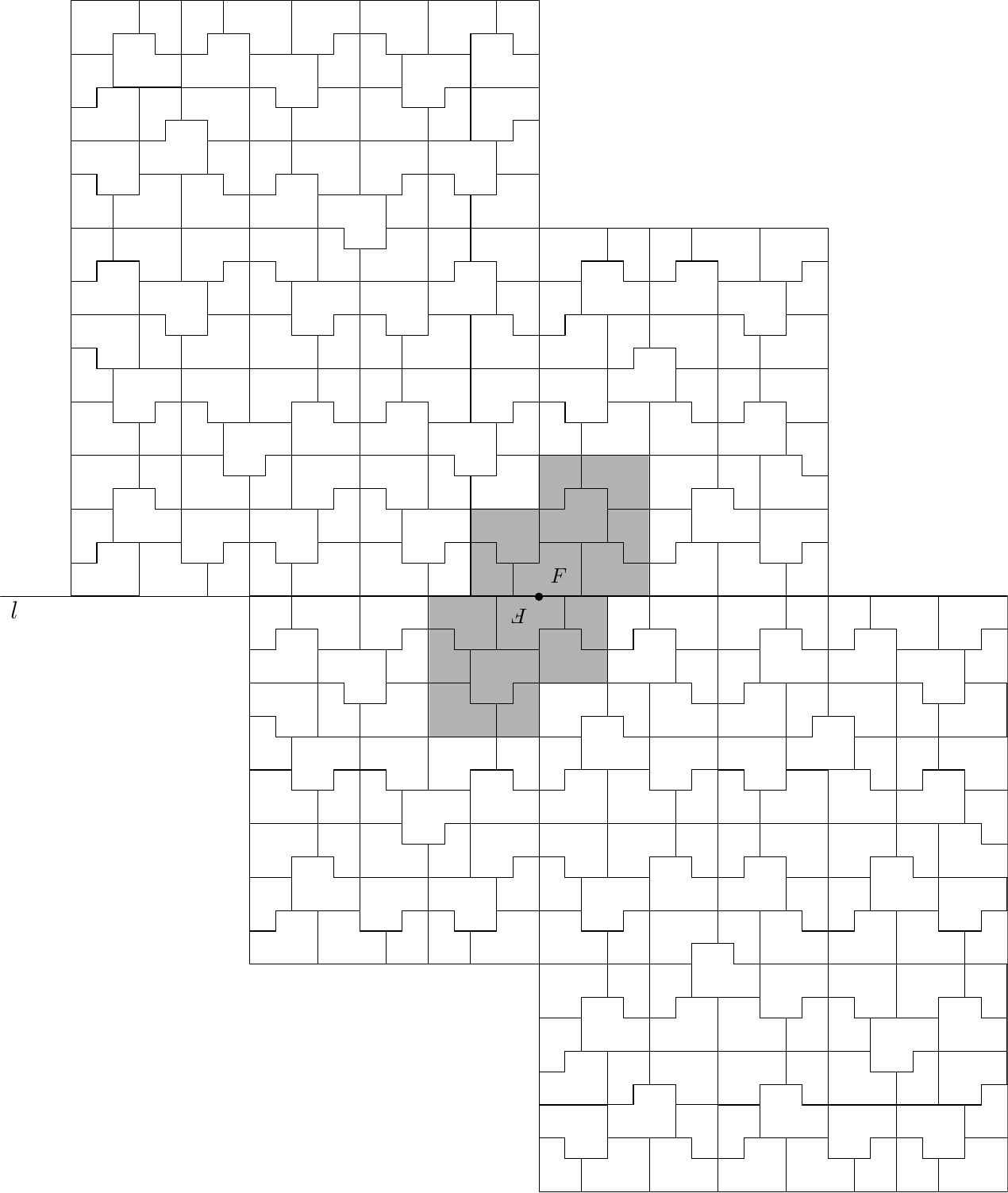}
        \end{center}
        \caption{\small The tiling $T$ includes the
          set $S_1=\F(\{F,\RR(F)\})$ marked grey.}\label{f62}
\end{figure}

The claim holds, since 
$T=\bigcup_{i=0}^{\infty}S_i$ and
$T$ is a fixed point of $\F$.
More specifically,
let $T'$ be an infinitely composable tiling of the plane that includes
almost all tiles from $T$.
Since $T=\bigcup_{i=0}^{\infty}S_i$, for some $i$ we have
\begin{equation}\label{eq3}
T\setminus S_i\subset T'.
\end{equation}

Assume first that $i=0$.
It is easy to verify by hand
that $T$ is the only tiling of the plane
that includes $T\setminus S_{0}$ and hence
$T'=T$. 

Assume that $i>0$.
We claim that in this case we have
\begin{equation}\label{eq4}
T\setminus S_{i-1}\subset \F^{-1}(T').
\end{equation}
To prove this claim,
apply $\sigma^{-4}$ to the inclusion~\eqref{eq3}.
We obtain
$$
\sigma^{-4}(T\setminus S_i)\subset \sigma^{-4}T'
$$
Since $T$ is the disjoint union of
$S_i$ and $T\setminus S_i$
we have
$$
\sigma^{-4}T=\sigma^{-4}S_i \sqcup \sigma^{-4}(T\setminus S_i)
$$
and hence
$$
\sigma^{-4}(T\setminus S_i)=\sigma^{-4}T\setminus \sigma^{-4}S_i.
$$
Thus
$$
\sigma^{-4}T\setminus \sigma^{-4}S_i\subset \sigma^{-4}T'.
$$
Applying $\HH^{-1}$ to this inclusion
we get~\eqref{eq4}, since $T$ is a fixed point of $\F$.

In the above arguments we implicitly
used Lemma~\ref{l0} several times.
To to show
that we may do that, we have to prove that all the tilings
of the form
$$
\sigma^{-j}S_i, \qquad \sigma^{-j}(T\setminus S_i),\qquad j=0,1,2,3
$$
are proper.
These tilings may be obtained from tilings
$S_0$ and $T\setminus S_0$ by applying substitution
$4i-j$ times and then applying a homothety.
If  $4i-j\ge 2$, then the tilings are proper,
since the can be obtained from some tilings by applying $\sigma^2$.
In the remaining case $i=1,j=3$
we  can  verify  by hand that the tilings
$\sigma^{-3}S_1$ and $\sigma^{-3}(T\setminus S_1)$ are proper.

  Repeating this trick $i$ times
  we can show
  that 
  $T\setminus S_{0}\subset \F^{-i}(T')$.
As we have seen, this implies 
  $\F^{-i}(T')=T$ and hence $T'=T$, as $T$ is a fixed point of $\F$.
\end{proof} 

\section*{Acknowledgments}
We are sincerely grateful to G.~Varouchas, 
M. Raskin and to anonymous referees for helpful remarks.
We thank A. Korotin for a careful reading 
a preliminary version of the text, which 
helped to improve essentially the exposition.

\appendix

\section{Appendix}
\subsection{The proof of Lemma~\ref{l0}}

(a) Every small tile $F$ from $S$ produces the same tile
  $(F\cup \text{the brother of }F)$ in  $\sigma^{-1}S$ and $\sigma^{-1}T$.
Every large tile $G$ from $S$ produces the same tile
$(F\cup \text{ the sister  of }G)$ in  $\sigma^{-1}S$ and $\sigma^{-1}T$,
if the sister  of $G$ is in $S$ (and hence in $T$).
Otherwise $S$ (and hence $T$) contains a tile $H$ that covers the
cavity of $G$. This tile is different from the sister of $G$.
Thus in this case the tile
$G$ produces itself in both $\sigma^{-1}S$ and $\sigma^{-1}T$.

(b) By item (a) both tilings $\sigma^{-1}S$ and
$\sigma^{-1}T$ are subsets of $\sigma^{-1}(S\cup T)$. Hence
$\sigma^{-1}S\cup \sigma^{-1}T\subset\sigma^{-1}(S\cup T)$.
This inclusion cannot be proper, as $\sigma^{-1}S\cup \sigma^{-1}T$
and $\sigma^{-1}(S\cup T)$ tile the same set.

\subsection{The proof of Proposition~\ref{standard-correct}}
(a) By induction: we will show that if a tiling $T$ is \arrows\ then  
so is $\sigma T$.
Assume that $T$ is an \arrows\ $d$-tiling. 
Color in the tiling $T$ all large 
and small tiles as shown on Fig.~\ref{orient}(b,c) (page~\pageref{orient}).
Such coloring will be called \emph{canonical}.
Assume that in the canonical coloring of $T$
the colors and orientations in all pairs of adjacent tiles match.
We have to show that the same holds for $\sigma T$.

To verify this, cut  
all large $d$-tiles from $T$, as shown on Fig.~\ref{orient}(a).
Then change the colored segments in the original 
canonical coloring using the following substitution:  
\begin{align*}
            \rsix & \to \rfive \\
            \rfive & \to \rfour \\
            \rfour & \to \rthree \\
            \rthree & \to \lfour\lsix.
\end{align*}        
The reverse arrows in $\lfour$ and  $\lsix$
mean that we reverse orientation. 
After that color the cut by colors
$\rfour$, $\rfive$, $\rsix$, as shown in Fig.~\ref{orient}(a).

As the transformation of colors of every segment 
does not depend on the tile it belongs to, it 
does not destroy the matching requirement. Therefore it remains to verify
that after transformation we get the canonical coloring of $\sigma T$.
This can be verified just by comparing 
Fig.~\ref{orient}(a),~\ref{orient}(b) and~\ref{orient}(c):
the transformation of colors applied to Fig.~\ref{orient}(c)
produces the coloring as in Fig.~\ref{orient}(b) and 
the transformation of colors applied to Fig.~\ref{orient}(b)
produces the coloring as in Fig.~\ref{orient}(a).

(b) Let $T$ be an  \arrows\  tiling of a convex set.
We have to show
that for every small tile $H\in T$ there is a  
large $G$ located as shown on Fig.~\ref{orient}(a).
Indeed, the cavity in $H$ formed by arrows $\rfive,\rsix$
is somehow filled by another tile $G$ in $T$. 
Notice that only large
tiles have a right angle with arrows $\rfive,\rsix$ thus
$G$ is a large tile. There is only one such angle in every large tile
and only one way to properly  attach a large tile to 
a small tile to fill the gap, namely the way shown in Fig.~\ref{orient}(a). 

(c)  Let $T$ be an \arrows\  tiling of a convex set.
We have to show that $\sigma^{-1}T$ is again an \arrows\ tiling. This is done
in a way similar to that in the proof of item (a).

The tiling $\sigma^{-1}T$ and its canonical 
coloring can be obtained from the canonical
coloring of $T$ in two steps:
erase sides $\lfour,\lfive,\lsix$ shared by each pair
(sister, brother),
to get the tiling $\sigma^{-1}T$ and then replace the colors using the map:
\begin{align*}
            \rsix\, \rfour & \to \lthree\\
            \rfive & \to \rsix \\
            \rfour & \to \rfive \\
            \rthree & \to \rfour. 
\end{align*}
Note that, after erasing sides shared by each pair
of small and large tile,
every occurrence of arrow $\rsix$
in a tile from $T$
is followed 
by an arrow $\rfour$  belonging to the same side of the same
tile. This is easy to verify looking at Fig.~\ref{orient}(b,c)
(the arrow $\rsix$ has two occurrences on sides of the large tile
and two occurrences on sides of the small tile; 
all they are followed by  $\rfour$
except for the side of a small tile that belongs to the cavity---but 
such sides have been erased).
An arrow $\rfour$ is replaced using the first rule, if it follows $\rsix$,
and using the third rule otherwise. 
In the  obtained coloring the colors and orientations match, 
as the transformation of every arrow does not depend
on the tile whose side it belongs to.

It remains to verify that the resulting coloring of $\sigma^{-1}T$
is indeed canonical. This can be verified
by comparing Fig.~\ref{orient}(a), Fig.~\ref{orient}(b) and
Fig.~\ref{orient}(c), as in the proof of item (a). 

\subsection{The proof of Proposition~\ref{l8}}

Items (a) and (c) 
follow from the following
        \begin{claima}
          Assume that $T$ is a finite supertile or an infinitely composable
          tiling. Assume that $H$ is any its tile (small or large).
          Then there is a unique sequence
          $H_0\subset H_1\subset H_2\subset\dots$
          that starts with $H$ and ends
          with $[T]$, if $T$ is a finite supertile, and is infinite,
if  $T$ is infinitely composable, and such that
          $H_i$ is either the son, or the daughter of $H_{i+1}$ for all
          $i$, and $\sss_d(H_i)\subset T$ for all $i$.
        \end{claima}  
        \begin{proof}
          For finite supertiles the statement can be proved by  induction
          on the level of $T$. If $T$ is a supertile
          of level $-1$ or 0, then the statement is obvious.
          Otherwise $T$ is a disjoint union  of supertiles $T'$ 
          and $T''$ of smaller levels, which tile the son and the daughter
          of $T$, respectively. Since  $T'$ 
and $T''$ are disjoint, we have either $H\in T'$, or 
$H\in T''$. In the first case the last but one tile
in the sought
sequence $H_0,H_1,H_2,\dots$
must be $[T']$ and the statement for $T$ follows from
the induction hypothesis for $T'$. Similarly,
in the second case the statement for $T$ follows from
the induction hypothesis for $T''$.

For infinitely composable tilings
we are unable to use similar arguments,
since the sequence $H_0, H_1, H_2,\dots$ must be infinite.
Let us first prove that such 
sequence exists.
For every $l$ consider the tiling 
$\sigma^{-l}T$ obtained from $T$ by $l$ compositions.
The tiling $T$ is  a
disjoint union of supertiles  $\sss_d(G)$
where $G\in \sigma^{-l}T$. 
Let $H_l$ denote the (unique) tile from $\sigma^{-l}T$
such that the supertile  $\sss_d(G)$ contains $H$.

We claim that for all $l$ the tile
$H_l$ is either the son, or the daughter of 
$H_{l+1}$, or $H_l$ coincides with $H_{l+1}$.
Indeed, the tiling $\sigma^{-l}T$ is the decomposition
of the tiling $\sigma^{-l-1}T$.
If  $H_{l+1}$ is a small tile in the tiling $\sigma^{-l-1}T$,
then its decomposition coincides with it, 
thus $H_{l+1}$ is in $\sigma^{-l}T$ 
and hence $H_l$ and $H_{l+1}$ coincide. Otherwise
the tiling 
$\sigma^{-l}T$ contains 
the result of decomposition of
$H_{l+1}$, that is, the 
son $F$ and the daughter $G$ of $H_{l+1}$. 
Since
$$
H\in \sss_d(H_{l+1})=S_d(F)\sqcup S_d(G),
$$
the tile $H$ belongs either to
$S_d(F)$, or to  $S_d(G)$.
In the first case  
$H_l$ must be equal to $F$ and otherwise $H_l=G$.

Removing repetitions from the sequence
$H_0, H_1,H_2, \dots$
we obtain the sought sequence of tiles.

Let us prove now that that such chain is unique.
Assume that there are two such chains
$H_0\subset H_1\subset H_2\subset\dots$
and 
$G_0\subset G_1\subset G_2\subset\dots$.
Let us show by induction on $n$ that $H_n=G_n$.
By assumption we have $H_0=G_0=H$.

Induction step: assume that  $H_n=G_n$.
By way of contradiction, assume
that $H_{n+1}\ne G_{n+1}$. Then the tile $H_n=G_n$
is the son of $H_{n+1}$ and the daughter of $G_{n+1}$
(or the other way around, but the other case in entirely similar).
Let $l$ stand for the level of the supertile
$\sss_{d}(H_n)=\sss_{d}(G_n)$.
The levels of supertiles $\sss_{d}(H_{n+1})$ and $\sss_{d}(G_{n+1})$
are $l+1$ and $l+2$ respectively.
By Lemma~\ref{l0}(a) both tilings 
$\sigma^{-l}\sss_{d}(H_{n+1})$, 
$\sigma^{-l}\sss_{d}(G_{n+1})$
are included into
$\sigma^{-l}T$.
The first tiling consists of $H_n$ and its sister.
The second one consists of $H_n$ and 
the son and the daughter of the brother of $H_n$.
Hence the tiling $\sigma^{-l}T$ contains 
the sister of $H_n$ and the son of the brother of $H_n$, which 
overlap  (see Fig.~\ref{f65} on page~\pageref{f65}).
\begin{figure}[t]
        \begin{center}
\includegraphics[scale=0.2]{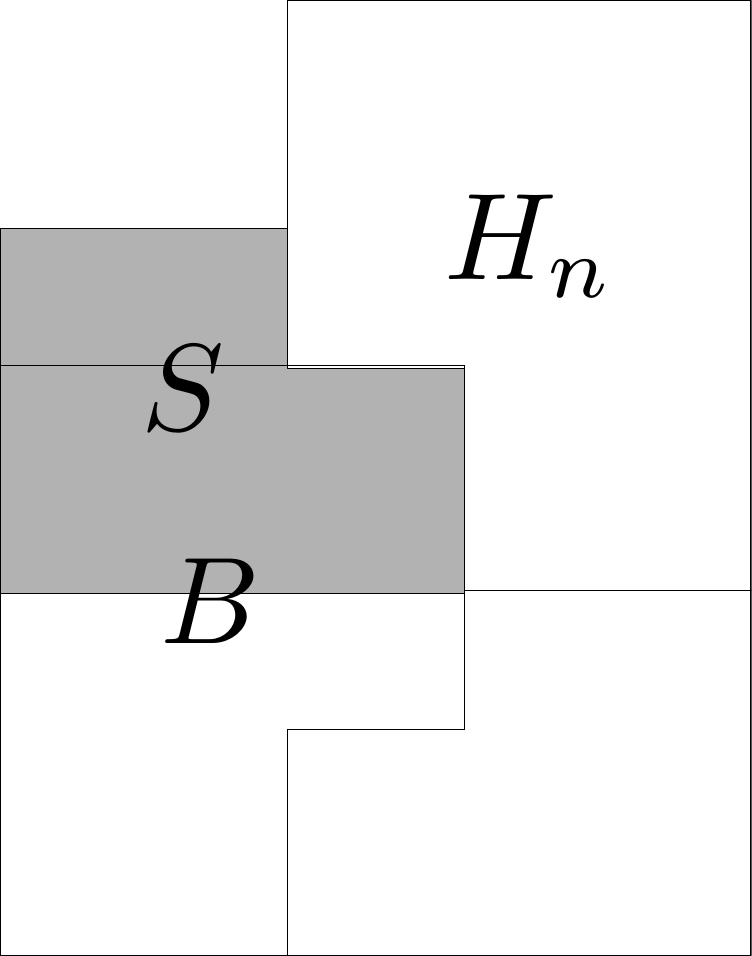}
        \end{center}
        \caption{\small The sister $S$ (marked grey)
          of the large tile $H_n$ overlaps
        with the son $B$ of the brother of $H_n$.}\label{f65}
\end{figure}
The obtained
contradiction proves that $H_{n+1}=G_{n+1}$.
\end{proof}

        (a)
        Let $T$ be an infinite supertile
        with representation $H_0\subset H_1\subset H_2\subset\dots$
        and $H$ any its tile.
Consider any $n$ such that $H\in \sss_d(H_n)$.
By the claim there is a chain of tiles
$G_0\subset G_1\subset \dots\subset G_l$ such that
$G_0=H$, $G_l=H_n$ and
$G_i$ is either the son, or the daughter of $G_{i+1}$ for all
$i$.
Then the sequence
$$
G_0, G_1,\dots G_l, H_{n+1},
H_{n+2}\dots
$$
is the sought representation of $T$.

The uniqueness part of the claim implies that such representation is unique.

(b) Let $F$ be any tile from $T$.
Then there are $k,l$ such that $F\in \sss_d(H_k)$ and  $F\in \sss_d(G_l)$.
Let $I_0,I_1,\dots,I_a$ denote the representation of $\sss_d(H_k)$
that starts with $F$
and $J_0,J_1,\dots,J_b$ the representation of $\sss_d(G_l)$ that starts with $F$.
Then both sequences of tiles
$$
I_0,I_1,\dots,I_a,H_{k+1},H_{k+2},\dots \quad 
J_0,J_1,\dots,J_b,G_{l+1},G_{l+2},\dots
$$
are   representations of $T$
that both start with $F$.
By item (a) these sequences coincide.
Note that $H_{k+i}$
is the  $(a+i)$th terms in the
first representation and $G_{l+j}$ is the 
$(b+j)$th in the
second representation (for all $i,j$).
Hence $H_{b+k+i}=G_{a+l+i}$ for all $i$.

(c) Consider the chain $H_0\subset H_1\subset H_2\subset\dots$
existing by the claim.
Then the union  $\cup_{n=1}^{\infty}\sss_d(H_n)$ is the sought supertile.
To prove uniqueness, notice that the representation of any supertile
$S$ satisfying the statement and starting with $H$
must satisfy the claim.


\subsection{Proof of Proposition~\ref{th-id}}

\emph{`If' part}.
Assume that
$\aaa=u\gamma$ and $\beta=v\gamma$ where $w(u)=w(v)$.
Then the supertiles with succinct representations $(H,u)$ and
$(G,v)$ have the same level (equal to $w(u)=w(v)$)
and hence are congruent. This implies that the infinite supertiles 
with succinct representations $(H,u\gamma)$ 
and $(G,v\gamma)$ are congruent as well. 

\emph{`Only if' part.}
We are given congruent infinite supertiles $T,S$
with succinct representations $(H,\alpha)$ and 
$(G,\beta)$, respectively.
W.l.o.g. we may assume that the tilings $T,S$
coincide (otherwise we apply 
to $H$ the isometry $f$ that maps  $T$ to 
$S$ and obtain another succinct representation
$(f(H),\alpha)$ of the supertile $S$). 

Thus we are given two different succinct 
representations $(H,\alpha)$ and $(G,\beta)$
of the same infinite supertile  
$S$. By Proposition~\ref{l8}(b) the corresponding
representations $(H=H_0), H_1,H_2,\dots$
and   $(G=G_0), G_1,G_2,\dots$
have the same tail, that is,
$H_{n+i}=G_{m+i}$ for some $m,n$ and all $i\ge 0$.
Thus  $\alpha=u\gamma$
and $\beta=v\gamma$ where 
$(H,u)$ and $(G,v)$ are succinct representations
of the $d$-supertile $\sss_d(H_n)=\sss_d(G_m)$,
and $\gamma$ is the infinite $s$-$l$-sequence
with
$$
\gamma_i=\begin{cases}
s,& \text{if $H_{n+i}=G_{m+i}$ is the daughter of $H_{n+i+1}=G_{m+i+1}$,}\\ 
l& \text{otherwise}.
\end{cases}
$$
We have $w(u)=w(v)$,
since $(H,u)$ and $(G,v)$ are succinct representations of the same supertile.


\begin{thebibliography}{99}


\bibitem{A}
Shigeki Akiyama, A note on aperiodic Ammann tiles, Discrete and
Computational Geometry 48 (2012) 702–710


\bibitem{ammann}
R. Ammann, B. Gr\"unbaum and
G.C. Shephard.
Aperiodic tiles.
\emph{Discrete and Computational Geometry,} v. 8 (1992). p.~1--25.

\bibitem{BG}
Michael Baake, Uwe Grimm, Aperiodic Order. Vol 1: A Mathematical
Invitation. Cambridge University Press 2013

\bibitem{brs}
Bruno Durand, Andrei Romashchenko, Alexander Shen.
Fixed-point tile sets and their applications
Journal of Computer and System Sciences,
 78:3 (2012) 731--764.

\bibitem{G}
  Chaim Goodman-Strauss, Matching Rules and Substitution Tilings,
  Annals of Mathematics 147 (1998) 181-223

\bibitem{GS}
Branko Gr\"unbaum, Geoffrey C. Shephard, Tilings and Patterns. Freeman,
New York 1987.

\bibitem{korotin}
 Alexander Korotin, Personal communication (2015).

\bibitem S
Boris Solomyak, Nonperiodicity implies unique composition for
self-similar translationally finite tilings, Discrete and Computational
Geometry 20 (1998) 265-279




\bibitem{sherer}
K. Scherer. A puzzling journey to the reptiles
and related animals. Privately published, 1987. 



%

\end{thebibliography}
\end{document}